\pdfoutput=1
\documentclass{article}

\usepackage[a4paper, margin=1.5in]{geometry}  
\usepackage{acro}
\usepackage{times} 
\usepackage[T1]{fontenc}      
\usepackage[utf8]{inputenc}   
\DeclareUnicodeCharacter{00A0}{ } 
\usepackage{microtype}        
\usepackage{bbm}              
\usepackage{amsmath}
\usepackage{amsthm}
\usepackage{amssymb}
\usepackage{commath}          
\usepackage{mathtools}
\usepackage{thmtools}
\usepackage{tcolorbox}
\usepackage{url}
\usepackage[authoryear,round]{natbib}
\usepackage{hyperref}
\usepackage[capitalise]{cleveref}
\usepackage{verbatim}               

\hypersetup{
    colorlinks=true,
    linkcolor=black,
    urlcolor=blue,
    citecolor=black,
    pdftitle={Toni Karvonen --- Asymptotic Bounds for Smoothness Parameter Estimates in Gaussian Process Interpolation},
}

\numberwithin{equation}{section}
\declaretheorem[Refname={Theorem,Theorems}]{theorem}
\numberwithin{theorem}{section} 
\declaretheorem[style=definition,numberlike=theorem,Refname={Definition,Definitions}]{definition}
\declaretheorem[style=definition,numberlike=theorem,Refname={Assumption,Assumptions}]{assumption}

\declaretheorem[style=definition,numberlike=theorem,Refname={Remark,Remarks}]{remark}
\declaretheorem[numberlike=theorem,Refname={Lemma,Lemmas}]{lemma}
\declaretheorem[name=Corollary,numberlike=theorem,Refname={Corollary,Corollaries}]{corollary}
\declaretheorem[name=Proposition,numberlike=theorem,Refname={Proposition,Propositions}]{proposition}

\newtheoremstyle{named}{}{}{\itshape}{}{\bfseries}{.}{.5em}{\thmnote{#3}}
\theoremstyle{named}
\newtheorem*{namedtheorem}{Theorem}

\DeclareMathOperator*{\argmin}{arg\,min}

\DeclarePairedDelimiterX\Set[2]{\lbrace}{\rbrace}%
{ #1 \,:\, #2 }                                         
\DeclarePairedDelimiterX\inprod[2]{\langle}{\rangle}%
{ #1 , #2 }                                             
\DeclarePairedDelimiter\floor{\lfloor}{\rfloor}         

\newcommand{\R}{\mathbb{R}} 
\newcommand{\N}{\mathbb{N}} 

\newcommand{\T}{\mathsf{T}} 

\newcommand{\ML}{\textup{ML}}
\newcommand{\CV}{\textup{CV}}

\DeclarePairedDelimiterX{\infdivx}[2]{(}{)}{%
  #1\;\delimsize\|\;#2%
}

\newcommand{\GP}{\textup{GP}}

\newcommand{\rev}[1]{#1}

\title{\textbf{Asymptotic Bounds for Smoothness Parameter Estimates in Gaussian Process Interpolation}\footnote{Journal reference: T.\ Karvonen (2023). Asymptotic bounds for smoothness parameter estimates in Gaussian process interpolation. \emph{SIAM/ASA Journal on Uncertainty Quantification}, 11(4):1225--1257.}}
\author{
  Toni Karvonen
  \vspace{0.2cm}
  \\
  \emph{Department of Mathematics and Statistics} \\
  \emph{University of Helsinki, Finland}
}

\date{}

\begin{document}

\maketitle

\renewcommand{\abstractname}{\vspace{-\baselineskip}}
\vspace{-0.5cm}
\begin{abstract}
  \noindent
  \textbf{Abstract.}
  It is common to model a deterministic response function, such as the output of a computer experiment, as a Gaussian process with a Matérn covariance kernel.
  The smoothness parameter of a Matérn kernel determines many important properties of the model in the large data limit, including the rate of convergence of the conditional mean to the response function.
  We prove that the maximum likelihood estimate of the smoothness parameter cannot asymptotically undersmooth the truth when the data are obtained on a fixed bounded subset of $\R^d$.
  That is, if the data-generating response function has Sobolev smoothness $\nu_0 > d/2$, then the smoothness parameter estimate cannot be asymptotically less than $\nu_0$.
  The lower bound is sharp.
  Additionally, we show that maximum likelihood estimation \rev{recovers the true} smoothness for a class of compactly supported self-similar functions.
  For cross-validation we prove an asymptotic lower bound $\nu_0 - d/2$, which however is unlikely to be sharp.
  The results are based on approximation theory in Sobolev spaces and some general theorems that restrict the set of values that the parameter estimators can take.
\end{abstract}

\section{Introduction}

Gaussian process interpolation is commonly used to approximate a deterministic response or data-generating function which may, for example, represent the output of a computer experiment~\citep{Sacks1989}.
A zero-mean Gaussian process is defined by a positive-definite covariance kernel $K_\theta$ with parameters $\theta \in \Theta$.
To ensure that Gaussian process interpolation yields a good approximation and reasonable quantification of uncertainty for the response function at unseen data locations, it is necessary to estimate the kernel parameters from the data.
Due to its flexibility and interpretability, the Matérn class of stationary covariance kernels is often preferred in applications~\citep{Stein1999}.
Let $\nu$, $\sigma$, and $\lambda$ be positive smoothness, magnitude, and scale parameters, respectively.
A Matérn kernel on $\R^d$ is defined as
\begin{equation} \label{eq:matern}
  K_\nu(x, y) = \sigma^2 c(\nu) \bigg( \frac{\sqrt{2\nu} \norm[0]{x - y}}{\lambda} \bigg)^{\!\nu} \mathcal{K}_{\nu} \bigg( \frac{\sqrt{2\nu} \norm[0]{x - y}}{\lambda} \bigg) \quad \text{ for } \quad x, y \in \R^d,
\end{equation}
where $\mathcal{K}_\nu$ is the modified Bessel function of the second kind of order~$\nu$ and $c(\nu)$ a positive $\nu$-dependent scaling factor.
Much is known about fixed-domain asymptotics of various estimators for the parameters $\sigma$ and $\lambda$ (as well as $\sigma^2 \lambda^{2\nu}$ for a fixed $\nu$) for Matérns and related kernels in both the \emph{Bayesian} setting where the response function is assumed to be a Gaussian process~\citep[e.g.,][]{Ying1991, Loh2005, Anderes2010, Bachoc2017} and the \emph{frequentist} setting where the response function is a fixed deterministic function~\citep{XuStein2017, Karvonen2020}.
As it defines assumed degree of differentiability of the response function and is microergdic~\citep[Section~6.2]{Stein1999}, the smoothness parameter $\nu$ is arguably the most important parameter of a Matérn kernel.
\rev{While it is common to fix $\nu$ beforehand, doing so is problematic when the smoothness of the response function is unknown (though data-driven estimation of $\sigma$ and $\lambda$ may overcome these issues):}
\begin{itemize}
\item If the model \emph{undersmooths} the truth (i.e., the response function is \rev{smoother} than assumed), uncertainty quantification is reliable, in the sense that the response function is contained in a credible set centered at the conditional mean for some fixed credible level \rev{[see~\eqref{eq:discussion-undersmoothing} and~\eqref{eq:discussion-credible-set}]}. 
  However, this comes at the cost of (likely) sub-optimal approximation accuracy.
\item If the model \emph{oversmooths} the truth, the approximation quality is best possible, in that \rev{the Narcowich--Ward--Wendland escape theorem [see~\eqref{eq:misspecified-rate}] guarantees} convergence of the conditional mean to the response function with a rate that is worst-case optimal in \rev{any Sobolev space which contains the response function}.
  However, uncertainty quantification may be unreliable.
\end{itemize}
The effects of under- and oversmoothing in the frequentist setting, as well as connections to the literature on construction of adaptive confidence and credible sets, are discussed in more detail in \Cref{sec:uq}.

\rev{Maximum likelihood estimation is perhaps the most popular data-driven approach to select the parameters of a Matérn kernel.}
It seems that the only theoretical results concerning maximum likelihood estimation of the smoothness parameter of a Matérn-type kernel have been obtained by \citet{Chen2021} and \citet{Petit2022}, \rev{who consider the periodic version of the Matérn kernel on $[0, 1]^d$~\citep[see][Section~6.7]{Stein1999} and show that maximum likelihood estimators are consistent}.
\citet{Szabo2015} and \citet{Knapik2016} have derived results for maximum likelihood smoothness estimation in a related white noise model. 
\rev{\citet{Loh2015, LohSunWen2021}; and \citet{LohSun2023} construct other smoothness estimators for the Matérn model whose consistency they prove under certain sampling schemes on $[0, 1]^d$.}
Other work on maximum likelihood estimation, as well as cross-validation, of parameters in Gaussian process and related models can be found in \citet{Bachoc2013, Szabo2013, RousseauSzabo2017, Bachoc2017, XuStein2017}; and \citet{HadjiSzabo2021}.
This article contains what appear to be the first theoretical results on maximum likelihood estimation (as well as cross-validation) of the smoothness parameter of the Matérn class on subsets of $\R^d$.
These results are described next for maximum likelihood estimation.
In short, we prove that (a) asymptotic undersmoothing is not possible and (b) smoothness is estimated \rev{consistently} for a class of compactly supported self-similar functions.

Let $f_0$ be a real-valued response function that is defined on a sufficiently regular bounded connected open subset $\Omega$ of $\R^d$, such as $\Omega = (0, 1)^d$, and suppose that $\{x_i\}_{i=1}^\infty$ is any quasi-uniform sequence (see \Cref{def:quasi-uniform}) of pairwise distinct points in $\Omega$.
Let
\begin{equation*}
  \hat{\nu}_\ML^{f_0}(X_n) = \argmin_{\nu \in \Theta} \big\{ f_0(X_n)^\T K_\nu(X_n)^{-1} f_0(X_n) + \log \det K_\nu(X_n) \big\},
\end{equation*}
where $\Theta \subset (0, \infty)$ is some interval, denote any maximum likelihood estimate of the smoothness parameter $\nu$ given the data vector $f_0(X_n) = (f_0(x_1), \ldots, f_0(x_n)) \in \R^n$ consisting of noiseless evaluations of $f_0$ at the points $X_n = \{x_i\}_{i=1}^n$.
Here $K_\nu(X_n) \in \R^{n \times n}$ is the kernel matrix for the Matérn kernel~\eqref{eq:matern} with elements $(K_\nu(X_n))_{ij} = K_{\nu}(x_i, x_j)$.
Let $\Theta = [\nu_\textup{min}, \nu_\textup{max}]$ for $0 < \nu_\textup{min} < \nu_\textup{max} < \infty$ be an interval large enough that the bounds and limits below can hold and suppose that the scaling factor $c(\nu)$ is bounded away from zero and infinity on $\Theta$ \rev{(e.g., being a continuous function of $\nu$)}.
\rev{The function $f_0$ is an element of $H^\alpha(\Omega)$, the Sobolev space of order $\alpha > d/2$, if it admits an extension $f_e \colon \R^d \to \R$ (i.e., $f_e|_\Omega = f_0$) whose Fourier transform $\smash{\widehat{f}_e}$ satisfies
\begin{equation*}
  \int_{\R^d} ( 1 + \norm[0]{\xi}^2)^\alpha \lvert \widehat{f}_e(\xi) \rvert^2 \dif \xi < \infty.
\end{equation*}
See \Cref{sec:sobolev} for more details on Sobolev spaces.
We prove that inclusion in a Sobolev space implies an asymptotic lower bound on the maximum likelihood estimate.}

\begin{namedtheorem}[No undersmoothing --- \Cref{thm:matern-main}]
Let $\nu_0 > d/2$.
If $f_0$ is an element of $H^{\nu_0}(\Omega)$, then 
\begin{equation} \label{eq:no-undersmoothing-intro}
  \liminf_{ n \to \infty } \hat{\nu}_\ML^{f_0}(X_n) \geq \nu_0.
\end{equation}
This bound is sharp in the sense that for every $\varepsilon > 0$ there is $f_0 \in H^{\nu_0}(\Omega)$ such that
\begin{equation*}
  \limsup_{ n \to \infty } \hat{\nu}_\ML^{f_0}(X_n) \leq \nu_0 + \varepsilon.
\end{equation*}
\end{namedtheorem}

\rev{Let $\nu(f_0) = \sup \Set{\nu > 0}{f_0 \in H^{\nu}(\Omega)}$ be the smoothness of $f_0$.
From~\eqref{eq:no-undersmoothing-intro} we get
\begin{equation*}
  \liminf_{ n \to \infty } \hat{\nu}_\ML^{f_0}(X_n) \geq \nu(f_0).
\end{equation*}
As satisfying as it would be, it does \emph{not} follow that $\smash{\hat{\nu}_\ML^{f_0}(X_n)} \to \nu(f_0)$.
In the context of density estimation and the Gaussian white noise model, it is well known that consistent estimation of smoothness and construction of adaptive confidence sets over Sobolev classes is impossible~\citep[Chapter~8]{PicardTribouley2000, GineNickl2010, Bull2012, Szabo2015, NicklSzabo2016, GineNickl2016}.
Additional self-similarity assumptions are needed to exclude ``inconvenient'' or ``deceptive'' functions whose smoothness cannot be estimated~\cite[see in particular][Section~3]{Szabo2015}.
In this vein, we say that $f_0$ is $\beta$-self-similar if it admits an extension $f_e$ such that
\begin{equation*}
  \sup_{\xi \in \R^d} \norm[0]{\xi}^{2\beta + d} \lvert \widehat{f}_e(\xi) \rvert^2 < \infty \quad \text{ and } \quad \int_{\norm[0]{\xi} \geq R} \lvert \widehat{f}_e(\xi) \rvert^2 \dif \xi \geq C R^{-2\beta}
\end{equation*}
for some positive $C$ and $R_0$ and all $R \geq R_0$.
The Fourier transform of a prototypical $f_e$ that satisfies these conditions is of order \smash{$\norm[0]{\xi}^{-(\beta+d/2)}$} as $\norm[0]{\xi} \to \infty$.
See \Cref{sec:self-similar-functions} for more details on self-similar functions.
We prove that maximum likelihood estimation of smoothness is consistent if $f_0$ is self-similar and supported on $\Omega$.
Because $\nu(f_0) = \beta$ if $f_0$ is $\beta$-self-similar (see \Cref{lemma:self-similar-inclusion}), in our context ``consistency'' simply means that the true smoothness of $f_0$ is recovered.}

\begin{namedtheorem}[Consistent estimation for self-similar functions --- \Cref{thm:matern-main-2}]
If $f_0$ is $\nu_0$-self-similar and has its support contained in $\Omega$, then
\begin{equation*}
  \lim_{n \to \infty} \hat{\nu}_\ML^{f_0}(X_n) = \nu_0.
\end{equation*}
\end{namedtheorem}

\rev{Except for the requirement that $f_0$ be supported in $\Omega$ in the latter theorem, the assumptions of our results are not particularly restrictive. More detailed discussion on the assumptions is deferred to \Cref{sec:main-matern}.}

Because the samples of a Gaussian process with a Matérn covariance kernel of smoothness $\nu_0$ have Sobolev smoothness $\nu_0$ but the reproducing kernel Hilbert space (RKHS) of the kernel is norm-equivalent to the Sobolev space of smoothness $\nu_0 + d/2$~\rev{(\citealp[e.g.,][]{Steinwart2019}; see \Cref{sec:sobolev,sec:sample-paths} for more details)}, the results indicate that maximum likelihood estimation recovers the smoothness for which $f_0$ ``resembles'' a sample from the corresponding Gaussian process rather than the smoothness for which $f_0$ is an element of the RKHS of the kernel.
When $f_0$ \emph{is} a zero-mean Gaussian process whose covariance kernel is a Matérn of smoothness $\nu_0$, then \Cref{cor:matern-samples}, a straightforward consequence of \Cref{thm:matern-main}, states that
\begin{equation*}
  \liminf_{ n \to \infty } \hat{\nu}_\ML^{f_0}(X_n) \geq \nu_0 \quad \text{almost surely}.
\end{equation*}
However, we emphasise that when $f_0$ is assumed a deterministic function, it does not have to be (nor should it be thought of as) a fixed sample path from a Gaussian process with some Matérn kernel---or in any other way related to some other stochastic process.
We also consider leave-one-out cross-validation estimation, for which we however can supply no upper bounds or results pertaining to self-similar functions. Even our lower bounds for cross-validation are likely to be off by $d/2$.

Our proofs make use of RKHSs and techniques from approximation theory in Sobolev spaces.
This particular approach has begun to gain popularity in various corners of the Gaussian process literature roughly within the past decade~\citep[e.g.,][]{Bull2011, StuartTeckentrup2018, Briol2019, WangTuoWu2020, Wynne2021}.
In \Cref{sec:general}, we begin by proving a number of general results (\Cref{thm:general-new,thm:upper-bounds-general,thm:general}) on parameter sets which cannot contain the parameter estimates.
The essence of these results is that maximum likelihood estimation and cross-validation attempt to find the simplest possible model, as quantified by the rate of decay of the conditional variance, that adequately explains the data.
\Cref{sec:matern} is then devoted to applying the general results to estimation of the Matérn smoothness parameter.
In \Cref{sec:smooth}, we discuss the application of \Cref{thm:general} to estimation of the scale parameter of infinitely smooth stationary kernels, such as the Gaussian kernel, though are unable to furnish any rigorous proofs.

\section{General Results} \label{sec:general}

This section reviews basic facts about Gaussian process interpolation and RKHSs and proves some general results on maximum likelihood estimation and cross-validation of covariance kernel parameters.

\subsection{Gaussian Process Interpolation}

Let $\Omega$ be an arbitrary infinite set which we call a \emph{domain} throughout this article.
By \emph{kernel} we mean a function $K_\theta \colon \Omega \times \Omega \to \R$ which is \emph{symmetric} and \emph{positive-definite}, in that
\begin{equation} \label{eq:pd-def}
  \sum_{i=1}^n \sum_{j=1}^n a_i a_j K_\theta(x_i, x_j) > 0
\end{equation}
for any $n \in \N$, any pairwise distinct $x_1, \ldots, x_n \in \Omega$, and any non-zero vector $(a_1, \ldots, a_n) \in \R^n$.
All kernels in this article are parametrised by some collection of parameters $\theta$ in a feasible parameter set~$\Theta$.
Equation~\eqref{eq:pd-def} implies that for any set $X = \{x_{i}\}_{i=1}^n \subset \Omega$ of $n$ pairwise distinct points the \emph{kernel matrix} $K_\theta(X) \in \R^{n \times n}$ with elements $(K_\theta(X))_{ij} = K_\theta(x_i, x_j)$ is positive-definite and thus invertible.
A Matérn kernel~\eqref{eq:matern} with any positive parameters is an example of a kernel on $\Omega = \R^d$.

That a stochastic process $f_\GP$ is a zero-mean Gaussian process with covariance $K_\theta$ implies that for any points $X$ the vector $(f_\GP(x_{1}), \ldots, f_\GP(x_{n}))$ is an $n$-dimensional normal random vector with mean zero and covariance $K_\theta(X)$.
Suppose that a \emph{deterministic} response function $f_0 \colon \Omega \to \R$ is modelled as a Gaussian process $f_\GP$.
Conditioning this process on the exact evaluations (i.e., data) $f_0(X) = (f_0(x_{1}), \ldots, f_0(x_{n})) \in \R^n$ of $f_0$ at some distinct points $X$ yields a conditional Gaussian process with the mean
\begin{equation} \label{eq:gp-mean}
  \mu_{\theta, f_0}(x \mid X) = \mathbb{E}[ f_\GP(x) \mid X, f_0(X) ] = K_\theta(x, X)^\T K_\theta(X)^{-1} f_0(X)
\end{equation}
and variance
\begin{equation} \label{eq:gp-variance}
  \mathbb{V}_\theta(x \mid X) = \mathrm{Var}[ f_\GP(x) \mid X, f_0(X) ] = K_\theta(x, x) - K_\theta(x, X)^\T K_\theta(X)^{-1} K_\theta(x, X),
\end{equation}
where $K_\theta(x, X)$ is an $n$-vector with elements $(K_\theta(x, X))_i = K_\theta(x, x_i)$.
Note that the variance can depend on $f_0$ only if $\theta$ is estimated from the data $f_0(X)$.

\subsection{Reproducing Kernel Hilbert Spaces}

Every symmetric positive-definite kernel $K_\theta \colon \Omega \times \Omega \to \R$ induces a unique \emph{reproducing kernel Hilbert space} (RKHS), $H(K_\theta)$.
This space consists of functions $f \colon \Omega \to \R$ and is equipped with an inner product $\inprod{\cdot}{\cdot}_\theta$ and the associated norm $\norm[0]{\cdot}_\theta$.
The terminology comes from the kernel~$K_\theta$ having the \emph{reproducing property}
\begin{equation} \label{eq:reproducing-property}
  \inprod{f}{K_\theta(\cdot, x)}_\theta = f(x) \quad \text{ for all } \quad f \in H(K_\theta) \: \text{ and } \: x \in \Omega.
\end{equation}
It may be difficult to determine if a given function is contained in the RKHS based merely on the algebraic form of the kernel and the function.
However, many general properties of the kernel, such as its continuity or degree of differentiability, are inherited by the functions in $H(K_\theta)$~\citep[e.g.,][Section~4.3]{Steinwart2008}.
Results on the relationship between RKHSs of stationary kernels whose Fourier transforms decay polynomially on $\R^d$ and Sobolev spaces are reviewed in \Cref{sec:matern}.
See the textbooks \citet{Berlinet2004} and \citet{Paulsen2016} for a wealth of additional information on RKHSs.

Most of our proofs rely on the connection between Gaussian process interpolation and optimal interpolation in an RKHS.
The history of this rather well known connection goes back at least to the work of \citet{KimeldorfWahba1970}.
We refer to \citet[Section~2.4]{Berlinet2004}; \citet{Scheuerer2013}; and \citet[Section~3]{Kanagawa2018} for recent reviews on the topic.
In short, the Gaussian process conditional mean equals the unique \emph{minimum-norm interpolant} in the RKHS and the conditional variance is the squared \emph{worst-case approximation error}.
That is,
\begin{equation} \label{eq:mean-is-min-norm}
  \mu_{\theta, f_0}(\cdot \mid X) = \argmin_{s \in H(K_\theta)} \Set[\big]{ \norm[0]{s}_\theta }{ s(x_i) = f_0(x_i) \text{ for every } i = 1, \ldots, n}
\end{equation}
and
\begin{equation} \label{eq:var-is-wce}
  \mathbb{V}_\theta(x \mid X) = \sup_{ \norm[0]{f}_\theta \leq 1} \abs[0]{ f(x) - \mu_{\theta, f}(x \mid X)}^2
\end{equation}
for every $x \in \Omega$.
Note that the correspondence~\eqref{eq:mean-is-min-norm} does not require that $f_0$ be an element of $H(K_\theta)$.
From~\eqref{eq:var-is-wce} it is straightforward to derive the fundamental error estimate
\begin{equation} \label{eq:rkhs-error}
  \abs[0]{ f(x) - \mu_{\theta, f}(x \mid X)} \leq \norm[0]{f}_\theta \mathbb{V}_\theta(x \mid X)^{1/2},
\end{equation}
which holds for every $f \in H(K_\theta)$ and $x \in \Omega$.

\subsection{Parameter Estimation in a General Setting} \label{sec:parameter-estimation}
 
Let $\smash{\{x_i\}_{i=1}^\infty}$ be a set of pairwise distinct points in $\Omega$ and denote $X_n = \smash{\{x_i\}_{i=1}^n}$.
Given evaluations of $f_0$ at points $X_n$, a \emph{maximum likelihood estimate}, $\smash{\hat{\theta}_{\ML}^{f_0}(X_n)}$, of $\theta$ is any minimiser of the function 
\begin{equation} \label{eq:ell-ml}
  \ell_{\ML}^{f_0}( \theta \mid X_n ) = f_0(X_n)^\T K_\theta(X_n)^{-1} f_0(X_n) + \log \det K_\theta(X_n),
\end{equation}
while a \emph{leave-one-out cross-validated estimate}, $\hat{\theta}_\CV^{f_0}(X_n)$, is any minimiser of the function
\begin{equation*}
  \ell_{\CV}^{f_0}( \theta \mid X_n ) = \sum_{i=1}^n \bigg[ \frac{(f_0(x_{i}) - \mu_{\theta, f_0}(x_i \mid X_n^i))^2}{\mathbb{V}_\theta( x_i \mid X_n^i )} + \log \mathbb{V}_\theta( x_{i} \mid X_n^i ) \bigg] ,
\end{equation*}
where we use superscripted $i$ to denote that the $i$th point has been removed (i.e., $X_n^i = X_n \setminus \{x_{i}\}$); see, for example, Section~5.4 in \citet{RasmussenWilliams2006}.
The following lemmas are useful.
Here we use the convention $\mathbb{V}_\theta( x \mid X_0 ) = \mathbb{V}_\theta( x \mid \emptyset ) = K_\theta(x, x)$.

\begin{lemma} \label{lemma:logdet}
  It holds that $\log \det K_\theta(X_n) = \sum_{i=1}^n \log \mathbb{V}_\theta(x_{i} \mid X_{i-1})$.
\end{lemma}
\begin{proof}
  The claim follows from straightforward iteration of the variance formula~\eqref{eq:gp-variance} and the block determinant identity
  \begin{equation*}
    \det \begin{pmatrix} a & b \\ b^\T & C \end{pmatrix} = \det(C) (a - b^\T C^{-1} b)
  \end{equation*}
  for any $a \in \R$, $b \in \R^{n-1}$ and any invertible $C \in \R^{(n-1) \times (n-1)}$.
\end{proof}

\begin{lemma} \label{lemma:minimum-norm-interpolation}
  For any $f_0 \colon \Omega \to \R$ we have 
  \begin{equation} \label{eq:mu-norm-explicit}
    f_0(X_n)^\T K_\theta(X_n)^{-1} f_0(X_n) = \norm[0]{\mu_{\theta, f_0}( \cdot \mid X_n)}_\theta^2.
  \end{equation}
  Moreover, if $f_0 \in H(K_\theta)$, then 
  \begin{equation*}
    f_0(X_n)^\T K_\theta(X_n)^{-1} f_0(X_n) = \norm[0]{\mu_{\theta, f_0}( \cdot \mid X_n)}_\theta^2 \leq \norm[0]{f_0}_\theta^2.
  \end{equation*}
\end{lemma}
\begin{proof}
  Equation~\eqref{eq:mu-norm-explicit} follows from the expression for the conditional mean in~\eqref{eq:gp-mean} and the fact, which is a consequence of the reproducing property in~\eqref{eq:reproducing-property}, that
  \begin{equation*}
    \norm[4]{\sum_{i=1}^n a_i K_\theta(\cdot, x_i) }_\theta^2 = \sum_{i=1}^n \sum_{j=1}^n a_i a_j \inprod{K_\theta(\cdot, x_i)}{K_\theta(\cdot, x_j)}_\theta = a^\T K_\theta(X_n) a
  \end{equation*}
  for any $a = (a_1, \ldots, a_n) \in \R^n$.
  The inequality is a consequence of the minimum-norm interpolation property in~\eqref{eq:mean-is-min-norm} and the fact that $f_0$ trivially interpolates itself.
\end{proof}

From the block matrix inversion formula one easily obtains the relatively well known~\citep[e.g.,][Section~4.2.2]{XuStein2017} expansion
\begin{equation} \label{eq:interpolant-norm-expansion}
  f_0(X_n)^\T K_\theta(X_n)^{-1} f_0(X_n) = \sum_{i=1}^n \frac{(f_0(x_i) - \mu_{\theta,f_0}(x_i \mid X_{i-1}))^2}{\mathbb{V}_\theta( x_i \mid X_{i-1})}.
\end{equation}
Applying~\eqref{eq:interpolant-norm-expansion} and \Cref{lemma:logdet} to~\eqref{eq:ell-ml} shows that the objective functions for maximum likelihood estimation and cross-validation are of similar form.
It should therefore be no surprise that the two parameter estimation methods share many properties~\citep[an additional interesting connection can be found in][]{FongHolmes2020}.

\subsubsection{Lower Bounds}

The following theorem yields lower bounds on smoothness parameter estimates.

\begin{theorem} \label{thm:general-new}
  Let $\Delta \subset \Theta$ and $\theta_0 \in \Theta$.
  \begin{enumerate}
  \item If $B$ is a set of real-valued functions on $\Omega$ such that
    \begin{equation} \label{eq:variance-decay-general-ml-new}
    \limsup_{n \to \infty} \sup_{ \theta \in \Delta } \sup_{f_0 \in B} \Bigg[ \norm[0]{\mu_{\theta_0, f_0}(\cdot \mid X_n)}_{\theta_0}^2 + \sum_{i=1}^n \log \frac{\mathbb{V}_{\theta_0}(x_{i} \mid X_{i-1})}{\mathbb{V}_{\theta}(x_{i} \mid X_{i-1})} \Bigg] < 0,
    \end{equation}
    then $\hat{\theta}_\ML^{f_0}(X_n) \notin \Delta$ for every $f_0 \in B$ when $n$ is sufficiently large.
  \item If $B$ is a set of real-valued functions on $\Omega$ such that
    \begin{equation} \label{eq:variance-decay-general-cv-new}
    \limsup_{n \to \infty} \sup_{ \theta \in \Delta } \sup_{f_0 \in B} \sum_{i=1}^n \bigg[ \frac{(f_0(x_{i}) - \mu_{\theta_0, f_0}(x_i \mid X_n^i))^2}{\mathbb{V}_{\theta_0}( x_i \mid X_n^i )} + \log \frac{\mathbb{V}_{\theta_0}(x_{i} \mid X_{n}^i)}{\mathbb{V}_{\theta}(x_{i} \mid X_{n}^{i})} \bigg] < 0,
    \end{equation}
    then $\hat{\theta}_\CV^{f_0}(X_n) \notin \Delta$ for every $f_0 \in B$ when $n$ is sufficiently large.
  \end{enumerate}
\end{theorem}
\begin{proof}
  Let us consider maximum likelihood estimation first.
  Let $\theta \in \Theta$. \Cref{lemma:logdet,lemma:minimum-norm-interpolation} yield
  \begin{equation*}
    \begin{split}
      \ell_\ML^{f_0}(\theta_0 \mid X_n) ={}& \norm[0]{\mu_{\theta_0, f_0}(\cdot \mid X_n)}_{\theta_0}^2 + \log \det K_{\theta_0}(X_n) \\
      \leq{}& \norm[0]{\mu_{\theta, f_0}(\cdot \mid X_n)}_{\theta}^2 + \norm[0]{\mu_{\theta_0, f_0}(\cdot \mid X_n)}_{\theta_0}^2 + \log \det K_{\theta_0}(X_n) \\
      ={}& \norm[0]{\mu_{\theta, f_0}(\cdot \mid X_n)}_{\theta}^2 + \norm[0]{\mu_{\theta_0, f_0}(\cdot \mid X_n)}_{\theta_0}^2 + \sum_{i=1}^{n} \log \mathbb{V}_{\theta_0}(x_{i} \mid X_{i-1}) \\
      ={}& \norm[0]{\mu_{\theta, f_0}(\cdot \mid X_n)}_{\theta}^2 + \norm[0]{\mu_{\theta_0, f_0}(\cdot \mid X_n)}_{\theta_0}^2 \\
      &+ \sum_{i=1}^{n} \bigg[  \log \mathbb{V}_{\theta}(x_{i} \mid X_{i-1}) + \log \frac{ \mathbb{V}_{\theta_0}(x_{i} \mid X_{i-1}) }{ \mathbb{V}_{\theta}(x_{i} \mid X_{i-1}) } \bigg] \\
      ={}& \ell_\ML^{f_0}(\theta \mid X_n) + \norm[0]{\mu_{\theta_0, f_0}(\cdot \mid X_n)}_{\theta_0}^2 + \sum_{i=1}^{n} \log \frac{ \mathbb{V}_{\theta_0}(x_{i} \mid X_{i-1}) }{ \mathbb{V}_{\theta}(x_{i} \mid X_{i-1}) }.
    \end{split}
  \end{equation*}
  It now follows from~\eqref{eq:variance-decay-general-ml-new} that $\ell_\ML^{f_0}(\theta_0 \mid X_n) < \inf_{\theta \in \Delta} \ell_\ML^{f_0}(\theta \mid X_n)$ for every $f_0 \in B$ when $n$ is sufficiently large, which means that the maximum likelihood estimate, being a minimiser of $\ell_\ML^{f_0}(\cdot \mid X_n)$, must be outside of $\Delta$ for all $f_0 \in B$ when $n$ is sufficiently large.
  The proof for cross-validation is analogous.
\end{proof}

\rev{The role of $\theta_0$ in \Cref{thm:general-new} is somewhat subtle. 
For $\theta = \theta_0$ the logarithmic terms in~\eqref{eq:variance-decay-general-ml-new} and~\eqref{eq:variance-decay-general-cv-new} are non-negative.
Because the non-logarithmic terms are always non-negative, this means that $\theta_0$ cannot be an element of $\Delta$.
Therefore one should select $\theta_0$ such that the set $\Delta$ can be made as large as possible.
When we consider smoothness estimation for Matérns in \Cref{sec:matern}, assumptions~\eqref{eq:variance-decay-general-ml-new} and~\eqref{eq:variance-decay-general-cv-new} are verified by taking $\theta_0$ and $\Delta$ such that [specifically, see~\eqref{eq:matern-proof-var-ratio} and~\eqref{eq:matern-proof-var-ratio-2}]
\begin{equation} \label{eq:ratio-goes-to-neg-inf}
  \lim_{n \to \infty} \sup_{\theta \in \Delta} \frac{ \mathbb{V}_{\theta_0}(x_{n} \mid X_{n-1}) }{ \mathbb{V}_{\theta}(x_{n} \mid X_{n-1}) } = 0,
\end{equation}
which states that the conditional variance should decay faster for the parameter $\theta_0$ than for any parameter $\theta \in \Delta$.
Because the conditional variance is the supremum over the unit ball of the RKHS by the equivalence in~\eqref{eq:var-is-wce}, the limit in~\eqref{eq:ratio-goes-to-neg-inf} implies that $H(K_{\theta_0})$ is \emph{essentially} smaller as a set than $H(K_\theta)$ for any $\theta \in \Delta$.
Note that~\eqref{eq:ratio-goes-to-neg-inf} is not necessary for~\eqref{eq:variance-decay-general-ml-new} and~\eqref{eq:variance-decay-general-cv-new} to hold.
For example, RKHSs of two Matérn kernels in~\eqref{eq:matern} with any positive scale parameters $\lambda_1$ and~$\lambda_2$ are norm-equivalent; see~\eqref{eq:rkhs-fourier-transform} and~\eqref{eq:matern-fourier}.
If one is estimating the scale parameter and $\Delta$ is an interval bounded away from zero and infinity, the ratio in~\eqref{eq:ratio-goes-to-neg-inf} cannot tend to zero.}

\rev{For maximum likelihood estimation it is not necessary to use \Cref{lemma:logdet} to decompose the determinant as a sum of variances.}
By writing
\begin{equation*}
  \log \det K_{\theta_0}(X_n) = \log \det K_{\theta}(X_n) + \log \det \big[ K_{\theta_0}(X_n) K_{\theta}(X_n)^{-1} \big],
\end{equation*} 
we could have replaced~\eqref{eq:variance-decay-general-ml-new} with the equivalent condition
\begin{equation} \label{eq:det-decay-general-ml-new}
  \limsup_{n \to \infty} \sup_{ \theta \in \Delta } \sup_{f_0 \in B} \Bigg[ \norm[0]{\mu_{\theta_0, f_0}(\cdot \mid X_n)}_{\theta_0}^2 + \log \frac{\det K_{\theta_0}(X_n)}{ \det K_\theta(X_n)} \Bigg] < 0.
\end{equation}
Because the kernel matrix determinant is known as the \emph{model complexity} and the \emph{data-fit term} $\norm[0]{\mu_{\theta_0, f_0}(\cdot \mid X_n)}_{\theta_0}$ quantifies how well the model fits the data, the condition~\eqref{eq:det-decay-general-ml-new} gives the following interpretation of \Cref{thm:general-new} for maximum likelihood estimation: If there is a parameter $\theta_0$ such that the model for this parameter fits the data sufficiently well (i.e., the data-fit is bounded or grows slowly) and each parameter in $\Delta$ corresponds to a model more complex than that for $\theta_0$ (i.e., the log-ratio of model complexities is sufficiently small or tends to negative infinity sufficiently fast), then the parameter estimate cannot be contained in $\Delta$.
That is, maximum likelihood estimation prefers simple models that fit the data well.

The following corollary is a specialisation of \Cref{thm:general-new} to a setting where $\Theta$ is an interval, which we take to be any connected subset of $\R$, and $\theta$ can be thought of as a smoothness parameter, so that $H(K_{\theta_1}) \subsetneq H(K_{\theta_2})$ whenever $\theta_1 > \theta_2$.\footnote{See \citet[Section~3.2]{Gualtierotti2015} for a general treatment of such \emph{contractive inclusions} of RKHSs.}
Under suitable conditions, the implication $\theta_0 \notin \Delta$ may then be expressed as an inequality that provides an asymptotic lower bound on the smoothness estimates.

\begin{corollary} \label{cor:general-interval-new}
  Let $\Theta \subset \R$ be an interval and $\theta_0 \in \Theta$.
  \begin{enumerate}
  \item If $B$ is a set of real-valued functions on $\Omega$ such that
    \begin{equation*}
    \limsup_{n \to \infty} \sup_{ \theta \leq \theta_1 } \sup_{f_0 \in B} \Bigg[ \norm[0]{\mu_{\theta_0, f_0}(\cdot \mid X_n)}_{\theta_0}^2 + \sum_{i=1}^n \log \frac{\mathbb{V}_{\theta_0}(x_{i} \mid X_{i-1})}{\mathbb{V}_{\theta}(x_{i} \mid X_{i-1})} \Bigg] < 0 \quad
    \end{equation*}
    for every $\theta_1 < \theta_0$, then 
    \begin{equation*}
      \liminf_{n \to \infty} \inf_{f_0 \in B} \hat{\theta}_\ML^{f_0}(X_n) \geq \theta_0.
    \end{equation*}
  \item If $B$ is a set of real-valued functions on $\Omega$ such that
    \begin{equation*}
    \limsup_{n \to \infty} \sup_{ \theta \leq \theta_1 } \sup_{f_0 \in B} \sum_{i=1}^n \bigg[ \frac{(f_0(x_{i}) - \mu_{\theta_0, f_0}(x_i \mid X_n^i))^2}{\mathbb{V}_{\theta_0}( x_i \mid X_n^i )} + \log \frac{\mathbb{V}_{\theta_0}(x_{i} \mid X_{n}^i)}{\mathbb{V}_{\theta}(x_{i} \mid X_{n}^{i})} \bigg] < 0
    \end{equation*}
    for every $\theta_1 < \theta_0$, then
    \begin{equation*}
      \liminf_{n \to \infty} \inf_{f_0 \in B} \hat{\theta}_\CV^{f_0}(X_n) \geq \theta_0.
    \end{equation*}
  \end{enumerate}
\end{corollary}
\begin{proof}
  Let $\Theta$ be an interval with endpoints $a_\Theta \leq b_\Theta$ that are possibly infinite.
  The claim follows by applying \Cref{thm:general-new} to $\Delta = \Theta \cap [a_\Theta, \theta_1]$ for each $\theta_1 < \theta_0$ and using the definition of the lower limit.
\end{proof}

In \Cref{sec:parameter-estimation} we shall assume that $B$ is a subset of $H(K_{\theta_0})$, which simplifies the role of $\theta_0$ but renders the results somewhat sub-optimal.

\subsubsection{Upper Bounds}

The following theorem yields upper bounds on smoothness parameter estimates.

\begin{theorem} \label{thm:upper-bounds-general}
  Let $\Sigma \subset \Theta$ and suppose that $B$ is a bounded subset of $H(K_{\theta_B})$ for some $\theta_B \in \Theta$.
  \begin{enumerate}
  \item If 
    \begin{equation} \label{eq:upper-bound-general-limsup-ml}
      \limsup_{n \to \infty} \mathbb{V}_{\theta_B}(x_n \mid X_{n-1}) < 1
    \end{equation}
    and
    \begin{equation} \label{eq:variance-decay-general-ml-new-upper}
      \liminf_{n \to \infty} \inf_{\theta \in \Sigma} \inf_{f_0 \in B} \ell_\ML^{f_0}(\theta \mid X_n) > -\infty,
    \end{equation}
    then $\hat{\theta}_\ML^{f_0}(X_n) \notin \Sigma$ for every $f_0 \in B$ when $n$ is sufficiently large.
  \item Let $b = \sup_{f_0 \in B} \norm[0]{f_0}_{\theta_B}$. If 
    \begin{equation} \label{eq:upper-bound-general-limsup-cv}
      \limsup_{n \to \infty} \max_{1 \leq i \leq n} \mathbb{V}_{\theta_B}(x_i \mid X_n^i) < \exp(-b^2)
    \end{equation}
    and
    \begin{equation} \label{eq:variance-decay-general-cv-new-upper}
      \liminf_{n \to \infty} \inf_{\theta \in \Sigma} \inf_{f_0 \in B} \ell_\CV^{f_0}(\theta \mid X_n) > -\infty,
    \end{equation}
    then $\hat{\theta}_\CV^{f_0}(X_n) \notin \Sigma$ for every $f_0 \in B$ when $n$ is sufficiently large.
  \end{enumerate}
\end{theorem}
\begin{proof}
  Let $b = \sup_{f_0 \in B} \norm[0]{f_0}_{\theta_B} < \infty$.
  Let us consider maximum likelihood estimation first.
  By \Cref{lemma:minimum-norm-interpolation} and~\eqref{eq:upper-bound-general-limsup-ml},
  \begin{equation*}
    \ell_\ML^{f_0}( \theta_B \mid X_n ) = \norm[0]{\mu_{\theta_B, f_0}(\cdot \mid X_n)}_{\theta_B}^2 + \sum_{i=1}^n \log \mathbb{V}_{\theta_B}(x_{i} \mid X_{i-1}) \leq b^2 + \sum_{i=1}^n \log \mathbb{V}_{\theta_B}(x_{i} \mid X_{i-1})
  \end{equation*}
  tends to negative infinity as $n \to \infty$ uniformly over $f_0 \in B$.
  It thus follows from~\eqref{eq:variance-decay-general-ml-new-upper} that 
  \begin{equation*}
    \sup_{f_0 \in B} \ell_\ML^{f_0}( \theta_B \mid X_n ) < \inf_{\theta \in \Sigma} \inf_{f_0 \in B} \ell_\ML^{f_0}(\theta \mid X_n)
  \end{equation*}
  for all sufficiently large $n$, which gives the claim for maximum likelihood estimation.
  The proof for cross-validation is analogous, except that now we use~\eqref{eq:rkhs-error} to get
  \begin{equation*}
    \begin{split}
      \ell_{\CV}^{f_0}( \theta_B \mid X_n ) &= \sum_{i=1}^n \bigg[ \frac{(f_0(x_{i}) - \mu_{\theta_B, f_0}(x_i \mid X_n^i))^2}{\mathbb{V}_{\theta_B}( x_i \mid X_n^i )} + \log \mathbb{V}_{\theta_B}( x_{i} \mid X_n^i ) \bigg] \\
      &\leq \sum_{i=1}^n \big[ b^2 + \log \mathbb{V}_{\theta_B}( x_{i} \mid X_n^i ) \big],
    \end{split}
  \end{equation*}
  which tends to negative infinity as $n \to \infty$ by~\eqref{eq:upper-bound-general-limsup-cv}.
\end{proof}

\rev{The assumptions~\eqref{eq:upper-bound-general-limsup-ml} and~\eqref{eq:upper-bound-general-limsup-cv} usually hold.
For instance, if $\Omega$ is a compact metric space, the sequence $\{ x_i \}_{i=1}^\infty$ is dense in $\Omega$, and $K_\theta$ is continuous, then $\sup_{ x \in \Omega } \mathbb{V}_\theta( x \mid X_n ) \to 0$ as $n \to \infty$.
That~\eqref{eq:upper-bound-general-limsup-cv} can be likely improved somewhat is discussed later in \Cref{rmk:cv-bounds-are-bad}.
The gist of \Cref{thm:upper-bounds-general} is in the interplay between $\theta_B$ and $\Sigma$: By~\eqref{eq:upper-bound-general-limsup-ml} and~\eqref{eq:upper-bound-general-limsup-cv}, both $\smash{\ell_\ML^{f_0}(\theta_B \mid X_n)}$ and $\smash{\ell_\CV^{f_0}(\theta_B \mid X_n)}$ tend to negative infinity, which may be interpreted as $\theta_B$ being a plausible parameter estimate.
The assumptions~\eqref{eq:variance-decay-general-ml-new-upper} and~\eqref{eq:variance-decay-general-cv-new-upper} then state that no parameter in $\Sigma$ is plausible, which limits the size of $\Delta$ by prohibiting $\theta \in \Delta$ such that $H(K_{\theta_B}) \subset H(K_\theta)$.
For if this inclusion were true and $\mathbb{V}_{\theta}(x_{n} \mid X_{n-1})$ tended to zero, both
\begin{equation*}
    \ell_\ML^{f_0}( \theta \mid X_n ) = \norm[0]{\mu_{\theta, f_0}(\cdot \mid X_n)}_{\theta}^2 + \sum_{i=1}^n \log \mathbb{V}_{\theta}(x_{i} \mid X_{i-1}) \leq \norm[0]{ f_0 }_\theta^2 + \sum_{i=1}^n \log \mathbb{V}_{\theta}(x_{i} \mid X_{i-1})
\end{equation*}
and $\smash{\ell_\CV^{f_0}( \theta \mid X_n )}$ would tend to negative infinity, thus violating~\eqref{eq:variance-decay-general-ml-new-upper} and~\eqref{eq:variance-decay-general-cv-new-upper}.} The following corollary provides a version of \Cref{thm:upper-bounds-general} adapted to intervals and complements \Cref{cor:general-interval-new}.

\begin{corollary} \label{cor:upper-bounds-interval}
Let $\Theta \subset \R$ be an interval and $\theta_0 \in \Theta$.
Suppose that $B$ is a bounded subset of $H(K_{\theta_B})$ for some $\theta_B \in \Theta$.
  \begin{enumerate}
  \item If 
    \begin{equation*}
      \limsup_{n \to \infty} \mathbb{V}_{\theta_B}(x_n \mid X_{n-1}) < 1 \quad \text{ and } \quad \liminf_{n \to \infty} \inf_{\theta \geq \theta_1} \inf_{f_0 \in B} \ell_\ML^{f_0}(\theta \mid X_n) > -\infty,
    \end{equation*}
    for every $\theta_1 > \theta_0$, then
    \begin{equation*}
      \limsup_{n \to \infty} \sup_{f_0 \in B} \hat{\theta}_\ML^{f_0}(X_n) \leq \theta_0.
    \end{equation*}
  \item Let $b = \sup_{f_0 \in B} \norm[0]{f_0}_{\theta_B}$. If 
    \begin{equation*}
      \limsup_{n \to \infty} \max_{1 \leq i \leq n} \mathbb{V}_{\theta_B}(x_i \mid X_n^i) < \exp(-b^2) \: \text{ and } \: \liminf_{n \to \infty} \inf_{\theta \geq \theta_1} \inf_{f_0 \in B} \ell_\CV^{f_0}(\theta \mid X_n) > -\infty,
    \end{equation*}
    for every $\theta_1 > \theta_0$, then
    \begin{equation*}
      \limsup_{n \to \infty} \sup_{f_0 \in B} \hat{\theta}_\CV^{f_0}(X_n) \leq \theta_0.
    \end{equation*}
  \end{enumerate}
\end{corollary}

Unlike the assumption that $B \subset H(K_{\theta_0})$ in the next section, the purpose of the assumption in \Cref{thm:upper-bounds-general} and \Cref{cor:upper-bounds-interval} that $B$ be a subset of \emph{some} RKHS is only to guarantee that the objective functions do not tend to negative infinity for all possible parameters.

\subsection{Parameter Estimation in an RKHS Setting} \label{sec:parameter-estimation}

By assuming that $B \subset H(K_{\theta_0})$ we obtain a weaker version of \Cref{thm:general-new}.

\begin{theorem} \label{thm:general}
  Let $\Delta \subset \Theta$ and $\theta_0 \in \Theta$.
  Suppose that $B$ is a bounded subset of $H(K_{\theta_0})$.
  \begin{enumerate}
  \item If
    \begin{equation} \label{eq:variance-decay-general-ml}
    \limsup_{n \to \infty} \, \sup_{ \theta \in \Delta } \frac{\mathbb{V}_{\theta_0}(x_{n} \mid X_{n-1})}{\mathbb{V}_{\theta}(x_{n} \mid X_{n-1})} < 1,
    \end{equation}
    then $\hat{\theta}_\ML(X_n) \notin \Delta$ for every $f_0 \in B$ when $n$ is sufficiently large.
  \item Let $b = \sup_{f_0 \in B} \norm[0]{f_0}_{\theta_0}$. If
    \begin{equation} \label{eq:variance-decay-general-cv}
      \limsup_{n \to \infty} \, \sup_{ \theta \in \Delta } \max_{ 1 \leq i \leq n} \frac{\mathbb{V}_{\theta_0}(x_{i} \mid X_{n}^i)}{\mathbb{V}_{\theta}(x_{i} \mid X_{n}^{i})} < \exp(-b^2), 
    \end{equation}
    then $\hat{\theta}_\CV(X_n) \notin \Delta$ for every $f_0 \in B$ when $n$ is sufficiently large.
  \end{enumerate}
\end{theorem}
\begin{proof}
  Let $b = \sup_{f_0 \in B} \norm[0]{f_0}_{\theta_0} < \infty$.
  Let us consider maximum likelihood estimation first.
  Because $B \subset H(K_{\theta_0})$, we get from \Cref{lemma:minimum-norm-interpolation} that
  \begin{equation*}
    \begin{split}
    \sup_{ \theta \in \Delta } \sup_{f_0 \in B} \Bigg[ \norm[0]{\mu_{\theta_0, f_0}(\cdot \mid X_n)}_{\theta_0}^2 + \sum_{i=1}^n &\log \frac{\mathbb{V}_{\theta_0}(x_{i} \mid X_{i-1})}{\mathbb{V}_{\theta}(x_{i} \mid X_{i-1})} \Bigg] \\
&\leq b^2 + \sum_{i=1}^n \sup_{\theta \in \Delta} \log \frac{\mathbb{V}_{\theta_0}(x_{i} \mid X_{i-1})}{\mathbb{V}_{\theta}(x_{i} \mid X_{i-1})},
    \end{split}
  \end{equation*}
  which tends to negative infinity as $n \to \infty$ by~\eqref{eq:variance-decay-general-ml}.
  Therefore~\eqref{eq:variance-decay-general-ml-new} holds, so that the claim follows from \Cref{thm:general-new}.
  Let us then consider cross-validation. 
  Because $B \subset H(K_{\theta_0})$, we may use~\eqref{eq:rkhs-error} to get
  \begin{equation*}
    \begin{split}
      \sup_{ \theta \in \Delta } \sum_{i=1}^n &\bigg[ \frac{(f_0(x_{i}) - \mu_{\theta_0, f_0}(x_i \mid X_n^i))^2}{\mathbb{V}_{\theta_0}( x_i \mid X_n^i )} + \log \frac{\mathbb{V}_{\theta_0}(x_{i} \mid X_{n}^i)}{\mathbb{V}_{\theta}(x_{i} \mid X_{n}^{i})} \bigg] \\
      &\leq\sum_{i=1}^n \bigg[ b^2 + \sup_{\theta \in \Delta} \log \frac{\mathbb{V}_{\theta_0}(x_{i} \mid X_{n}^i)}{\mathbb{V}_{\theta}(x_{i} \mid X_{n}^{i})} \bigg] \leq n \bigg[ b^2 + \sup_{ \theta \in \Delta } \max_{ 1 \leq i \leq n} \log \frac{\mathbb{V}_{\theta_0}(x_{i} \mid X_{n}^i)}{\mathbb{V}_{\theta}(x_{i} \mid X_{n}^{i})} \bigg],
    \end{split}
  \end{equation*}
  which tends to negative infinity as $n \to \infty$ by~\eqref{eq:variance-decay-general-cv}.
  Therefore~\eqref{eq:variance-decay-general-cv-new} holds, so that the claim follows from \Cref{thm:general-new}.
\end{proof}

\begin{remark} \label{rmk:cv-bounds-are-bad}
  Suppose for simplicity that $B = \{f_0\}$.
  It is likely that~\eqref{eq:variance-decay-general-cv}, and similarly~\eqref{eq:upper-bound-general-limsup-cv},  can be improved to requiring simply that the upper limit be less than one.
  For had we used~\eqref{eq:interpolant-norm-expansion} and~\eqref{eq:rkhs-error} and proceeded as we did in the case of cross-validation, we would have arrived at the similar assumption
  \begin{equation*}
    \limsup_{n \to \infty} \, \sup_{ \theta \in \Delta } \frac{\mathbb{V}_{\theta_0}(x_{n} \mid X_{n-1})}{\mathbb{V}_{\theta}(x_{n} \mid X_{n-1})} < \exp(-\norm[0]{f_0}_{\theta_0}^2)
  \end{equation*}
  for maximum likelihood estimation from 
  \begin{equation*}
    \begin{split}
    \sup_{ \theta \in \Delta } \Bigg[ &\norm[0]{\mu_{\theta_0, f_0}(\cdot \mid X_n)}_{\theta_0}^2 + \sum_{i=1}^n \log \frac{\mathbb{V}_{\theta_0}(x_{i} \mid X_{i-1})}{\mathbb{V}_{\theta}(x_{i} \mid X_{i-1})} \Bigg] \\
&= \sup_{ \theta \in \Delta } \sum_{i=1}^n \bigg[ \frac{(f_0(x_i) - \mu_{\theta,f_0}(x_i \mid X_{i-1}))^2}{\mathbb{V}_\theta( x_i \mid X_{i-1})} + \log \frac{\mathbb{V}_{\theta_0}(x_{i} \mid X_{i-1})}{\mathbb{V}_{\theta}(x_{i} \mid X_{i-1})} \bigg] \\
&\leq n \norm[0]{f_0}_{\theta_0}^2 + \sum_{i=1}^n \sup_{ \theta \in \Delta } \log \frac{\mathbb{V}_{\theta_0}(x_{i} \mid X_{i-1})}{\mathbb{V}_{\theta}(x_{i} \mid X_{i-1})}.
    \end{split}
  \end{equation*}
  This indicates that using~\eqref{eq:rkhs-error} should be avoided. 
It is straightforward to improve~\eqref{eq:rkhs-error} to $\abs[0]{f(x) - \mu_{\theta,f}(x \mid X)} \leq \norm[0]{f - \mu_{\theta,f}(\cdot \mid X)}_\theta \mathbb{V}_\theta(x \mid X)^{1/2}$. However, controlling the RKHS norm $\norm[0]{f - \mu_{\theta,f}(\cdot \mid X)}_\theta$ is challenging.
\end{remark}

To see that \Cref{thm:general} is weaker than \Cref{thm:general-new}, observe that assumptions~\eqref{eq:variance-decay-general-ml-new} and~\eqref{eq:variance-decay-general-cv-new} can hold even when
  \begin{equation*}
    \norm[0]{\mu_{\theta_0, f_0}(\cdot \mid X_n)}_{\theta_0}^2 \to \infty \quad \text{ or } \quad \frac{(f_0(x_{i}) - \mu_{\theta_0, f_0}(x_i \mid X_n^i))^2}{\mathbb{V}_{\theta_0}( x_i \mid X_n^i )} \to \infty,
  \end{equation*}
which, as we saw in the proof of \Cref{thm:general}, cannot happen if $B \subset H(K_{\theta_0})$.
This weakness of \Cref{thm:general} is explained by the fact that $B \subset H(K_{\theta_0})$ is a ``wrong'' assumption to make because the samples of a Gaussian process with covariance kernel $K_\theta$ are \emph{not} elements of $H(K_\theta)$ but of a somewhat larger RKHS~(\citealt{Driscoll1973}; we discuss this more in \Cref{sec:driscoll,sec:sample-paths}). 
That is, maximum likelihood estimation and cross-validation do not attempt to find $\theta_0$ such that $B \subset H(K_{\theta_0})$ but $\theta_0$ for which the elements of $B$ resemble, in some sense, the samples of a Gaussian process with covariance kernel $K_{\theta_0}$.
We shall see this phenomenon in action in \Cref{sec:matern} because for Matérn kernels on $\R^d$ the samples have $d/2$ orders of smoothness less than the RKHS.
By applying \Cref{thm:general} to the setting where $\Theta$ is an interval we obtain the following corollary.

\begin{corollary} \label{cor:general-interval}
  Let $\Theta \subset \R$ be an interval and $\theta_0 \in \Theta$.
  Suppose that $B$ is a bounded subset of $H(K_{\theta_0})$.
  \begin{enumerate}
  \item If
      \begin{equation} \label{eq:variance-decay-interval-ml}
    \limsup_{n \to \infty} \, \sup_{ \theta \leq \theta_1 } \frac{\mathbb{V}_{\theta_0}(x_{n} \mid X_{n-1})}{\mathbb{V}_{\theta}(x_{n} \mid X_{n-1})} < 1 \quad \text{ for every } \quad \theta_1 < \theta_0,
      \end{equation}
      then
      \begin{equation*}
        \liminf_{n \to \infty} \inf_{f_0 \in B} \hat{\theta}_\ML^{f_0}(X_n) \geq \theta_0.
      \end{equation*}
  \item Let $b = \sup_{f_0 \in B} \norm[0]{f_0}_{\theta_0}$. If
      \begin{equation} \label{eq:variance-decay-interval-cv}
        \limsup_{n \to \infty} \, \sup_{ \theta \leq \theta_1 } \max_{ 1 \leq i \leq n} \frac{\mathbb{V}_{\theta_0}(x_{i} \mid X_{n}^i)}{\mathbb{V}_{\theta}(x_{i} \mid X_{n}^{i})} < \exp(-b^2) \quad \text{ for every } \quad \theta_1 < \theta_0,
      \end{equation}
      then
      \begin{equation*}
        \liminf_{n \to \infty} \inf_{f_0 \in B} \hat{\theta}_\CV^{f_0}(X_n) \geq \theta_0.
      \end{equation*}
  \end{enumerate}
\end{corollary}

\subsection{On Driscoll's Theorem} \label{sec:driscoll}

The determinantal condition~\eqref{eq:det-decay-general-ml-new} has a connection to sample path properties of Gaussian processes that is worth elucidating.
Because the product of two positive-definite matrices has positive eigenvalues, the inequality of arithmetic and geometric means yields
\begin{equation} \label{eq:AM-GM}
  \det \big[ K_{\theta_0}(X_n) K_{\theta}(X_n)^{-1} \big] \leq \bigg( \frac{1}{n} \mathrm{tr}\big[ K_{\theta_0}(X_n) K_{\theta}(X_n)^{-1} \big] \bigg)^n.
\end{equation}
From~\eqref{eq:det-decay-general-ml-new} and~\eqref{eq:AM-GM} we obtain the following variant of \Cref{thm:general} for maximum likelihood estimation.

\begin{theorem} \label{thm:general-trace}
  Let $\Delta \subset \Theta$ and $\theta_0 \in \Theta$.
  If $B$ is a set of real-valued functions on $\Omega$ such that
  \begin{equation} \label{eq:general-trace-condition}
    \limsup_{n \to \infty} \sup_{ \theta \in \Delta } \sup_{f_0 \in B} \Bigg[ \norm[0]{\mu_{\theta_0, f_0}(\cdot \mid X_n)}_{\theta_0}^2 + n \log \bigg( \frac{1}{n} \mathrm{tr}\big[ K_{\theta_0}(X_n) K_{\theta}(X_n)^{-1} \big] \bigg) \Bigg] < 0,
  \end{equation}
  then $\hat{\theta}_\ML^{f_0}(X_n) \notin \Delta$ for every $f_0 \in B$ when $n$ is sufficiently large.
\end{theorem}

Suppose that $\Omega$ is a separable metric space, that $\smash{\{x_i\}_{i=1}^\infty}$ is dense in $\Omega$, and that the kernels $K_{\theta_0}$ and $K_\theta$ are continuous.
\citet[Theorem~3]{Driscoll1973} has proved that, under certain additional assumptions,
\begin{equation} \label{eq:driscoll-condition}
  \lim_{ n \to \infty } \mathrm{tr}\big[ K_{\theta_0}(X_n) K_{\theta}(X_n)^{-1} \big] < \infty
\end{equation}
if and only if the samples from a Gaussian process with covariance kernel $K_{\theta_0}$ are contained in $H(K_\theta)$ with probability one.
In particular, setting $\theta = \theta_0$ shows that the samples are not contained in the RKHS of $K_{\theta_0}$.
See \citet{LukicBeder2001} for the equivalence of~\eqref{eq:driscoll-condition} to a \emph{nuclear dominance} condition between $K_\theta$ and $K_{\theta_0}$.
If the set $B$ is such that $\norm[0]{\mu_{\theta_0, f_0}(\cdot \mid X_n)}_{\theta_0}^2$ does not grow faster than linearly in $n$, which is the case if $B$ is a bounded subset of $H(K_{\theta_0})$, then~\eqref{eq:general-trace-condition} is implied by the following version of Driscoll's condition~\eqref{eq:driscoll-condition} that is uniform over $\theta \in \Delta$:
\begin{equation} \label{eq:driscoll-uniform}
  \limsup_{ n \to \infty } \sup_{\theta \in \Delta} \mathrm{tr}\big[ K_{\theta_0}(X_n) K_{\theta}(X_n)^{-1} \big] < C < \infty
\end{equation}
for some $C > 0$.
If~\eqref{eq:driscoll-uniform} holds, then
\begin{equation*} 
  \limsup_{ n \to \infty } \sup_{\theta \in \Delta} \bigg[ n \log \bigg( \frac{1}{n} \mathrm{tr}\big[ K_{\theta_0}(X_n) K_{\theta}(X_n)^{-1} \big] \bigg) \bigg] \leq n \log \frac{C}{n} = n \log C - n \log n,
\end{equation*}
which implies~\eqref{eq:general-trace-condition} if the RKHS norm is assumed to grow at most linearly.

\section{Smoothness Estimation for Matérns} \label{sec:matern}

In this section we apply the results of \Cref{sec:general} to estimation of the smoothness parameter $\nu$ of the Matérn class in~\eqref{eq:matern}.
We fix the positive magnitude and scale parameters $\sigma$ and $\lambda$ and write the Matérn kernel of smoothness $\nu > 0$ as
\begin{equation*}
  K_\nu(x, y) = \Phi_\nu( x - y ),
\end{equation*}
where the function
\begin{equation} \label{eq:Phi-matern}
  \Phi_\nu(z) = \sigma^2 c(\nu) \bigg( \frac{\sqrt{2\nu} \norm[0]{z} }{\lambda} \bigg)^{\!\nu} \mathcal{K}_\nu \bigg( \frac{\sqrt{2\nu} \norm[0]{z} }{\lambda} \bigg)
\end{equation}
is defined on $\R^d$.
Recall that $\mathcal{K}_\nu$ is the modified Bessel function of the second kind of order~$\nu$.
In most results of this section we shall consider smoothness estimation over the bounded interval $\Theta = [\nu_\textup{min}, \nu_\textup{max}]$ for $0 < \nu_\textup{min} \leq \nu_\textup{max} < \infty$ and employ the following assumptions on the positive scaling factor $c(\nu)$:
\begin{subequations}
  \begin{equation} \label{eq:matern-factor-assumption-lower}
    \inf_{ \nu \in [\nu_\textup{min}, \nu_\textup{max}] } c(\nu) > 0,
  \end{equation}
  and
  \begin{equation} \label{eq:matern-factor-assumption-upper}
    \sup_{ \nu \in [\nu_\textup{min}, \nu_\textup{max}] } c(\nu) < \infty.
  \end{equation}
\end{subequations}
These assumptions hold if the scaling factor is a continuous function of $\nu$.
Let $\Gamma$ denote the Gamma function.
The scaling factor $c(\nu) = 2^{1-\nu} / \Gamma(\nu)$, which is typically used because it ensures that $K_\nu$ tends pointwise to the Gaussian kernel
\begin{equation*}
  K(x, y) = \sigma^2 \exp\bigg(\!-\frac{\norm[0]{x-y}^2}{2\lambda^2} \bigg)
\end{equation*}
as $\nu \to \infty$~\citep[pp.\@~49--50]{Stein1999}, is obviously continuous.

We use $\lesssim_n$ (with $\gtrsim_n$ defined analogously) to denote an inequality that holds up to a constant factor for all $n \in \N$.
That is, $a_n \lesssim_n b_n$ means that there is a non-negative constant $C$ such that $a_n \leq C b_n$ for all $n$.
We write $a_n \asymp_n b_n$ if $a_n \lesssim_n b_n$ and $a_n \gtrsim_n b_n$.

\subsection{Sobolev Spaces} \label{sec:sobolev}

The Fourier transform of $f \in L^2(\R^d)$ is defined as $\widehat{f}(\xi) = \int_{\R^d} f(x) e^{-\mathrm{i} x^\T \xi} \dif x$.
For $\alpha > 0$, the \emph{Sobolev space} $H^\alpha(\R^d)$ consists of functions $f \in L^2(\R^d)$ such that
\begin{equation} \label{eq:sobolev-norm}
  \norm[0]{f}_{H^\alpha(\R^d)}^2 = \int_{\R^d} ( 1 + \norm[0]{\xi}^2)^\alpha \lvert \widehat{f}(\xi) \rvert^2 \dif \xi < \infty.
\end{equation}
By the Sobolev embedding theorem, assuming that $\alpha > d/2$ ensures that every element of $H^\alpha(\R^d)$ can be uniquely identified with a continuous function, in which case $H^\alpha(\R^d)$ can be interpreted as a space of functions rather than of their equivalence classes.
If $\alpha$ is an integer, the Sobolev space consists of functions whose weak derivatives up to order $\alpha$ exist and are in $L^2(\R^d)$.
Moreover, every function in $H^\alpha(\R^d)$ is $\lfloor \alpha - d/2 \rfloor$ times differentiable in the classical sense.
On a subset $\Omega$ of $\R^d$ the Sobolev space $H^{\alpha}(\Omega)$ is defined as the set of functions $f \colon \Omega \to \R$ for which there exists an \emph{extension} $f_e \in H^\alpha(\R^d)$ such that $f_e|_\Omega = f$.
The norm of $H^\alpha(\Omega)$ is
\begin{equation} \label{eq:sobolev-space-restricted}
  \norm[0]{f}_{H^\alpha(\Omega)} = \min\Set{ \norm[0]{f_e}_{H^\alpha(\R^d)}}{ f_e \in H^{\alpha}(\R^d) \text{ and } f_e|_\Omega = f}.
\end{equation}
We shall formulate all our auxiliary results in $H^\alpha(\R^d)$ and use $H^\alpha(\Omega)$ only in what we consider to be main results of this section.

The RKHS of any stationary kernel of the form $K_\theta(x, y) = \Phi_\theta(x - y)$ for an integrable and continuous $\Phi_\theta \colon \R^d \to \R$ can be expressed in terms of the Fourier transform of $\Phi_\theta$~\citep[Theorem~10.12]{Wendland2005}.
Namely, $H(K_\theta)$ contains those square-integrable and continuous functions $f \colon \R^d \to \R$ for which
\begin{equation} \label{eq:rkhs-fourier-transform}
  \norm[0]{f}_\theta^2 = \int_{\R^d} \frac{\lvert \widehat{f}(\xi) \rvert^2}{\widehat{\Phi}_\theta(\xi)} \dif \xi < \infty.
\end{equation}
The function $\Phi_\nu$ in~\eqref{eq:Phi-matern}, which defines the Matérn class, has the Fourier transform
\begin{equation} \label{eq:matern-fourier}
  \widehat{\Phi}_\nu(\xi) = \frac{1}{C_\nu} \bigg( \frac{2\nu}{\lambda^2} + \norm[0]{\xi}^2 \bigg)^{-(\nu + d/2)} \quad \text{ with } \quad C_\nu = \frac{\pi^{d/2}}{\sigma^2 c(\nu) 2^{\nu - 1} \Gamma(\nu + d/2)} \bigg( \frac{\lambda^2}{2\nu} \bigg)^\nu.
\end{equation}
Therefore the norm $\norm[0]{\cdot}_\nu$ of the Matérn RKHS $H(K_\nu)$ is
\begin{equation} \label{eq:matern-norm}
  \norm[0]{f}_\nu^2 = C_\nu \int_{\R^d} \bigg( \frac{2\nu}{\lambda^2} + \norm[0]{\xi}^2 \bigg)^{\nu + d/2} \lvert \widehat{f}(\xi) \rvert^2 \dif \xi.
\end{equation}
It is straightforward to compute that
\begin{equation} \label{eq:matern-norm-equivalence}
  C_\nu \norm[0]{f}_{H^{\nu + d/2}(\R^d)}^2 \min\{1, b_\nu \} \leq \norm[0]{f}_\nu^2 \leq C_\nu \max\{1, b_\nu\} \norm[0]{f}_{H^{\nu + d/2}(\R^d)}^2,
\end{equation}
where $b_\nu = (2\nu / \lambda^2)^{\nu + d/2}$.
This shows (as is well known) that $H(K_\nu)$ is \emph{norm-equivalent} to the Sobolev space $H^{\nu + d/2}(\R^d)$, which is to say that $H(K_\nu) = H^{\nu + d/2}(\R^d)$ as sets and the norms $\norm[0]{\cdot}_\nu$ and $\norm[0]{\cdot}_{H^{\nu + d/2}(\R^d)}$ are equivalent.

\subsection{Self-Similar Functions} \label{sec:self-similar-functions}

Lower bounds on smoothness parameter estimates that we prove apply to any Sobolev function.
But to obtain upper bounds we need to work with a class of self-similar functions that we define as follows.

\begin{definition} \label{def:self-similar}
Let $\beta > 0$.
We say that a function $f \in L^2(\R^d)$ is $\beta$-\emph{self-similar} if
\begin{equation} \label{eq:self-similar-upper}
  \sup_{\xi \in \R^d} \norm[0]{\xi}^{2\beta + d} \lvert \widehat{f}(\xi) \rvert^2 < \infty
\end{equation}
and there are positive constants $C$ and $R_0$ such that 
\begin{equation} \label{eq:self-similar-lower}
  \int_{\norm[0]{\xi} \geq R} \lvert \widehat{f}(\xi) \rvert^2 \dif \xi \geq C R^{-2\beta}
\end{equation}
for all $R \geq R_0$.
A function on $\Omega$ is $\beta$-self-similar if it has a $\beta$-self-similar extension.
\end{definition}

\rev{
As discussed in the introduction, self-similarity assumptions are routinely used in the literature on non-parametric statistics to exclude ``inconvenient'' or ``deceptive'' functions whose smoothness cannot be estimated~\citep{Bull2012, Szabo2015, NicklSzabo2016}.
See in particular Section~3 in \citet{Szabo2015}.
Functions such that
\begin{equation*}
  C_1 \norm[0]{ \xi }^{-(\beta + d/2)} \leq \abs[0]{\widehat{f}(\xi)} \leq C_2 \norm[0]{ \xi }^{-(\beta + d/2)}
\end{equation*}
for some positive $C_1$ and $C_2$ and all $\xi$ outside some ball centered at the origin are prototypical examples of $\beta$-self-similar functions because
\begin{equation*}
  \sup_{ \xi \in \R^d} \norm[0]{\xi}^{2\beta + d} \lvert \widehat{f}(\xi) \rvert^2 \leq C_2^2
\end{equation*}
and
\begin{equation} \label{eq:ss-proto-integral-lower}
  \int_{\norm[0]{\xi} \geq R} \lvert \widehat{f}(\xi) \rvert^2 \dif \xi \geq C_1^2 \int_{\norm[0]{\xi} \geq R} \norm[0]{ \xi }^{-2\beta - d} \dif \xi = C_1^2 C_d R^{-2\beta}
\end{equation}
for a certain positive constant $C_d$.}

\begin{lemma} \label{lemma:self-similar-inclusion}
  If $f \in L^2(\R^d)$ is $\beta$-self-similar, then $f \in H^\alpha(\R^d)$ if $\alpha < \beta$ and $f \notin H^\alpha(\R^d)$ if $\alpha > \beta$.
\end{lemma}
\begin{proof}
By~\eqref{eq:self-similar-upper}, there is a non-negative constant~$c$ such that
\begin{equation*}
  \norm[0]{f}_{H^\alpha(\R^d)}^2 = \int_{\R^d} ( 1 + \norm[0]{\xi}^2)^\alpha \lvert \widehat{f}(\xi) \rvert^2 \dif \xi \leq c \int_{\R^d} ( 1 + \norm[0]{\xi}^2)^\alpha \norm[0]{\xi}^{-2\beta - d}  \dif \xi,
\end{equation*}
which is finite if $\alpha < \beta$.
Therefore a $\beta$-self-similar function is in $H^\alpha(\R^d)$ for every $\alpha < \beta$.
On the other hand, from~\eqref{eq:self-similar-lower} we get
\begin{equation*}
  \begin{split}
  \norm[0]{f}_{H^\alpha(\R^d)}^2 = \int_{\R^d} ( 1 + \norm[0]{\xi}^2)^\alpha \lvert \widehat{f}(\xi) \rvert^2 \dif \xi &\geq \int_{\norm[0]{\xi} \geq R} ( 1 + \norm[0]{\xi}^2)^\alpha \lvert \widehat{f}(\xi) \rvert ^2 \dif \xi \\
  &\geq 2^\alpha R^{2\alpha} \int_{\norm[0]{\xi} \geq R} \lvert \widehat{f}(\xi) \rvert^2 \dif \xi \\
  &\geq C 2^\alpha R^{2(\alpha-\beta)}
  \end{split}
\end{equation*}
for every $R \geq \max\{1, R_0\}$.
Therefore $f \notin H^\alpha(\R^d)$ if $\alpha > \beta$.
\end{proof}

We shall work with self-similar functions that are supported in a given open set.
Let $\Omega \subset \R^d$ be an open set and define
\begin{equation*}
  H_\textup{ss}^\beta(\Omega) = \Set{ f \in L^2(\R^d) }{ \text{$f$ is $\beta$-self-similar and the support of $f$ is contained in $\Omega$}}.
\end{equation*}
It seems likely that the requirement that the functions be supported in $\Omega$ is not necessary in our results; see the discussion in \Cref{sec:main-matern}.
The following basic construction shows that $H_\textup{ss}^\beta(\Omega)$ is non-empty.

\begin{lemma} \label{lemma:ft-lower-bound}
Let $\Omega \subset \R^d$ be an open set.
For every $\beta > d/2$ there is $f \in L^2(\R^d)$ such that the support of $f$ is contained in $\Omega$ and
\begin{equation} \label{eq:ft-lower-bound-f0}
  C_1 (1 + \norm[0]{\xi}^2)^{-(\beta + d/2)} \leq \lvert \widehat{f}(\xi) \rvert^2 \leq C_2 (1 + \norm[0]{\xi}^2)^{-(\beta + d/2)}
\end{equation}
for some positive $C_1$ and $C_2$ and all $\xi \in \R^d$.
This function is in $H_\textup{ss}^{\beta}(\Omega)$ and an element of $H^{\alpha}(\R^d)$ if and only if $\alpha < \beta$.
\end{lemma}
\begin{proof}
We may assume without loss of generality that $\Omega$ is any open ball.
By first setting $\nu = \beta/2 - d/4 > 0$ and then selecting $\sigma$ and $\lambda$ properly in~\eqref{eq:Phi-matern} and~\eqref{eq:matern-fourier}, we obtain a function~$g$ with the Fourier transform
\begin{equation*}
  \widehat{g}(\xi) = (1 + \norm[0]{\xi}^2)^{-(\beta + d/2)/2}.
\end{equation*}
There exists a non-negative \emph{bump function} $\phi \colon \R^d \to \R$ that (i) is supported on the unit ball, (ii) satisfies $\sup_{x \in \R^d} \phi(x) = \phi(0) = 1$ and (iii) is infinitely differentiable and hence has a Fourier transform which decays faster than any polynomial.
The standard example of such a bump function is given by
\begin{equation} \label{eq:bump-function}
  \phi(x) = \exp\bigg(\! -\frac{1}{1 - \norm[0]{x}^2}\bigg) \: \text{ if } \: \norm[0]{x} < 1 \quad \text{ and } \quad \phi(x) = 0 \: \text{ if } \: \norm[0]{x} \geq 1.
\end{equation}
Because $\phi$ is radial, its Fourier transform is real-valued and therefore the function $\phi_2 = \phi * \phi$ has non-negative Fourier transform by the convolution theorem.
Because $\phi$ is supported on the unit ball, so are $\phi_2$ and $f \coloneqq \phi_2 g$.
It remains to show that $f$ satisfies other requirements in the lemma.

First, because the Fourier transform of $\phi$ decays faster than any polynomial, there is $C_\phi > 0$ such that 
\begin{equation*}
  \widehat{\phi}_2(\xi) = \widehat{\phi}(\xi)^2 \leq C_\phi (1 + \norm[0]{\xi}^2)^{-(\beta + d/2)/2}
\end{equation*}
for all $\xi \in \R^d$.
From $\norm[0]{\xi} \leq \norm[0]{\xi - \omega} + \norm[0]{\omega}$ it follows that
\begin{equation*}
\begin{split}
  \abs[0]{\widehat{f}(\xi)} ={}& \int_{\R^d} \widehat{\phi}_2(\omega) \widehat{g}(\xi - \omega) \dif \omega \\
  \leq{}& C_\phi \int_{\R^d} (1 + \norm[0]{\omega}^2)^{-(\beta + d/2)/2} (1 + \norm[0]{\xi - \omega}^2)^{-(\beta + d/2)/2} \dif \omega \\
  \leq{}& C_\phi \int_{\norm[0]{\xi - \omega} \leq \frac{1}{2} \norm[0]{\xi}} \bigg( 1 + \frac{\norm[0]{\xi}^2}{4} \bigg)^{-(\beta + d/2 )/2} (1 + \norm[0]{\xi - \omega}^2)^{-(\beta + d/2)/2} \dif \omega \\
  &+ C_\phi \int_{\norm[0]{\xi - \omega} \geq \frac{1}{2} \norm[0]{\xi}} (1 + \norm[0]{\omega}^2)^{-(\beta + d/2)/2} \bigg( 1 + \frac{\norm[0]{\xi}^2}{4} \bigg)^{-(\beta + d/2)/2} \dif \omega \\
  \leq{}& 2^{\beta+d/2} C_\phi \bigg[ \int_{\R^d} (1 + \norm[0]{\omega}^2)^{-(\beta + d/2)/2} \dif \omega \bigg] ( 1 + \norm[0]{\xi}^2 )^{-(\beta + d/2)/2},
\end{split}
\end{equation*}
which establishes the upper bound in~\eqref{eq:ft-lower-bound-f0}.

Because $\smash{\widehat{\phi}_2}$ is continuous and $\smash{\widehat{\phi}_2(0)} = \smash{\int_{\R^d} \phi_2(x) \dif x > 0}$, there is a positive constant $\delta$ such that $c_\phi \coloneqq \smash{\min_{\norm[0]{\omega} \leq \delta} \widehat{\phi}_2(\omega) > 0}$.
Therefore
\begin{equation*}
  \begin{split}
    \abs[0]{\widehat{f}(\xi)} = \int_{\R^d} \widehat{\phi}_2(\xi - \omega) \widehat{g}(\omega) \dif \omega \geq{}& \int_{\norm[0]{\xi - \omega} \leq \delta} \widehat{\phi}_2(\xi - \omega) (1 + \norm[0]{\omega}^2)^{-(\beta + d/2)/2} \dif \omega \\
    \geq{}& c_\phi \int_{\norm[0]{\xi - \omega} \leq \delta} (1 + \norm[0]{\omega}^2)^{-(\beta + d/2)/2} \dif \omega.
  \end{split}
\end{equation*}
Let $C_{d,\delta}$ be the volume of a $d$-dimensional $\delta$-ball.
For $\norm[0]{\xi} \geq \delta$ we get
\begin{equation*}
  \begin{split}
    \abs[0]{\widehat{f}(\xi)} &\geq c_\phi \int_{\norm[0]{\xi - \omega} \leq \delta} (1 + \norm[0]{\omega}^2)^{-(\beta + d/2)/2} \dif \omega \\
    &\geq c_\phi \int_{\norm[0]{\xi - \omega} \leq \delta} (1 + (\norm[0]{\xi} + \delta)^2)^{-(\beta + d/2)/2} \dif \omega \\
    &\geq C_{d,\delta} \, c_\phi (1 + 4\norm[0]{\xi}^2)^{-(\beta + d/2)/2} \\
    &\geq 2^{-(\beta+d/2)} C_{d, \delta} \, c_\phi (1 + \norm[0]{\xi}^2)^{-(\beta + d/2)/2}.
  \end{split}
\end{equation*}
The lower bound in~\eqref{eq:ft-lower-bound-f0} follows from this estimate and $\inf_{\norm[0]{\xi} \leq \delta} \abs[0]{\widehat{f}(\xi)} > 0$.

It is clear that $f$ satisfies~\eqref{eq:self-similar-upper}, while~\eqref{eq:self-similar-lower} follows from a computation similar to~\eqref{eq:ss-proto-integral-lower}.
Finally, that $f \in H^{\alpha}(\R^d)$ if and only if $\alpha < \beta$ is a consequence of~\eqref{eq:ft-lower-bound-f0} and the fact that 
\begin{equation*}
  \int_{\R^d} (1 + \norm[0]{\xi}^2)^{\alpha} ( 1 + \norm[0]{\xi}^2)^{-(\beta + d/2)} \dif \xi < \infty
\end{equation*}
if and only if $\alpha < \beta$.
\end{proof}

\subsection{Smoothness Estimation}

Let $\Omega \subset \R^d$ be a bounded set.
The \emph{fill-distance} $h_{n,\Omega}$ of the points $X_n = \{x_i\}_{i=1}^n$ is
\begin{equation*}
  h_{n, \Omega} = \sup_{x \in \Omega} \mathrm{dist}(x, X_n) = \sup_{x \in \Omega} \min_{i=1,\ldots,n} \norm[0]{x - x_i}
\end{equation*}
and their \emph{separation radius} $q_n$ is
\begin{equation*}
  q_n = \frac{1}{2} \min_{i \neq j} \norm[0]{x_i - x_j}.
\end{equation*}
The fill-distance equals the radius of the largest ball in $\Omega$ which does not contain any of the points in $X_n$, while an open ball with radius $q_n$ can contain at most one point in $X_n$.

\begin{definition}[Quasi-uniformity] \label{def:quasi-uniform}
  Suppose that $\Omega \subset \R^d$ is bounded.
  A point sequence $\{x_i\}_{i=1}^\infty \subset \Omega$ is \emph{quasi-uniform} on $\Omega$ if there is a positive constant $c_\textup{qu}$ such that
  \begin{equation} \label{eq:quasi-uniform}
    q_n \leq h_{n, \Omega} \leq c_\textup{qu} q_n \quad \text{ for every } \quad n \in \N.
  \end{equation}
\end{definition}

Quasi-uniform points cover the domain $\Omega$ in a sufficiently uniform manner, in that the ratio between the distance of the two nearest points in $X_n$ and the diameter of the largest empty region in $\Omega$ which does not contain any of the points in $X_n$ remains bounded from above and below.
Quasi-uniformity implies that~\citep[e.g.,][Proposition~14.1]{Wendland2005}
\begin{equation} \label{eq:quasi-uniform-rate}
  q_n \asymp_n h_{n,\Omega} \asymp_n n^{-1/d}.
\end{equation}
Although the quasi-uniformity assumption is not satisfied by random points, extensions of our results for random points may be possible by following \citet{KriegSonnleitner2022} and considering $L^p(\Omega)$-norms of the distance function $\mathrm{dist}(\cdot, X_n)$ for $p < \infty$.
The following assumption on regularity of the domain $\Omega$ is needed in some of our results.

\begin{assumption} \label{ass:domain}
  The domain $\Omega \subset \R^d$ is a bounded connected open set which satisfies an interior cone condition and has a Lipschitz boundary.
\end{assumption}

See Definition~3.6 in \citet{Wendland2005} for the interior cone condition and p.\@~189 in \citet{Stein1970} for the definition of a Lipschitz boundary.
The former of the assumptions prohibits the existence of pinch points on the boundary of $\Omega$ by requiring that each $x \in \Omega$ be a vertex of a cone contained in $\Omega$ while the latter prescribes that the boundary of $\Omega$ is sufficiently regular.
Standard domains, such as $(0, 1)^d$ and open bounded convex sets, satisfy \Cref{ass:domain}.

\subsubsection{Variance Bounds}

\Cref{prop:matern-upper,prop:matern-lower} provide upper and lower bounds on the conditional variance for Matérn kernels.
Although both propositions are well known in the literature~\citep[e.g.,][]{Novak1988, NovakTriebel2006}, we include a full (and fairly thorough) proof of the latter proposition because we need to keep track of the constants that appear in the bounds.

\begin{proposition} \label{prop:matern-upper}
  Suppose that $\Omega \subset \R^d$ satisfies \Cref{ass:domain} and that the points $\{x_i\}_{i=1}^\infty$ are such that $h_{n, \Omega} \lesssim_n n^{-1/d}$.
  Then
  \begin{equation*}
    \sup_{x \in \Omega} \mathbb{V}_\nu(x \mid X_n) \lesssim_n n^{-2\nu / d}
  \end{equation*}
  for every $\nu > 0$.
\end{proposition}
\begin{proof}
  This result is a direct consequence of Corollary~4.1 in \citet{Arcangeli2007} and the norm-equivalence of $H(K_\nu)$ and $H^{\nu + d/2}(\R^d)$.
  A slightly less general version would follow from Proposition~3.6 in \citet{WendlandRieger2005}.
\end{proof}

\begin{proposition} \label{prop:matern-lower}
  Suppose that $\Omega \subset \R^d$ is bounded, that $c(\nu)$ satisfies~\eqref{eq:matern-factor-assumption-lower}, and that the points $\{x_i\}_{i=1}^\infty$ are quasi-uniform.
  Then
  \begin{equation*}
    \inf_{\nu \in [\nu_\textup{min}, \tau]} \min_{1 \leq i \leq n} \mathbb{V}_\nu( x_i \mid X_n^i) \gtrsim_n n^{-2\tau  / d}
  \end{equation*}
  for every $\tau \in [\nu_\textup{min}, \nu_\textup{max}]$.
\end{proposition}
\begin{proof}
  Let $\tau \in [\nu_\textup{min}, \nu_\textup{max}]$ and $\nu \in [\nu_\textup{min}, \tau]$.
  Recall from~\eqref{eq:var-is-wce} that
\begin{equation*}
  \mathbb{V}_\nu(x_i \mid X_n^i) = \sup_{ \norm[0]{f}_\nu \leq 1} \abs[0]{ f(x) - \mu_{\nu, f}(x_i \mid X_n^i)}^2.
\end{equation*}
From this expression it follows that $\mathbb{V}_\nu( x_i \mid X_n^i) \geq \abs[0]{f(x_i)}^2$ if $f \in H(K_\nu) = H^{\nu + d/2}(\R^d)$ is a function such that $\norm[0]{f}_\nu \leq 1$ and $f(x) = 0$ for every $x \in X_n^i$ since, as can be seen from~\eqref{eq:gp-mean}, in this case $\mu_{\nu,f}(\cdot \mid X_n^i) \equiv 0$.
We now construct such a function.

  Let $\phi$ be the bump function in~\eqref{eq:bump-function}.
  For any $q > 0$, define the function $g \colon \R^d \to \R$ via
  \begin{equation} \label{eq:gx-fun-def}
    g(x) = \phi\bigg( \frac{x - x_i}{q} \bigg).
  \end{equation}
  This function is an element of $H^{\nu + d/2}(\R^d)$ for every $\nu > 0$ because the Fourier transform of~$\phi$ decays faster than any polynomial, which implies that the Sobolev norm in~\eqref{eq:sobolev-norm} is finite.
  Suppose that $q \leq 1$.
  Equation~\eqref{eq:matern-norm} and the scaling properties of the Fourier transform give
  \begin{equation*}
    \begin{split}
      \norm[0]{g}_\nu^2 &= C_\nu \int_{\R^d} \bigg( \frac{2\nu}{\lambda^2} + \norm[0]{\xi}^2 \bigg)^{\nu + d/2} \lvert \widehat{g}(\xi) \rvert^2 \dif \xi \\
      &\leq C_\nu q^{2d} \int_{\R^d} \bigg( \frac{2\nu}{\lambda^2} + \norm[0]{\xi}^2 \bigg)^{\nu + d/2} \lvert \widehat{\phi}(q\xi) \rvert^2 \dif \xi.
      \end{split}
  \end{equation*}
  A change of variables and some basic estimates based on $q \leq 1$ and $\nu \leq \tau$ then yield
  \begin{equation*}
    \begin{split}
      \norm[0]{g}_\nu^2 &\leq C_\nu q^d \int_{\R^d} \bigg( \frac{2\nu}{\lambda^2} + \frac{\norm[0]{\xi}^2}{q^2} \bigg)^{\nu + d/2} \lvert \widehat{\phi}(\xi) \rvert^2 \dif \xi \\
      &= C_\nu \lambda^{-(2\nu + d)} q^{-2\nu} \int_{\R^d} ( 2\nu q^2 + \lambda^2 \norm[0]{\xi}^2 )^{\nu + d/2} \lvert \widehat{\phi}(\xi) \rvert^2 \dif \xi \\
      &\leq C_\nu \lambda^{-(2\nu + d)} q^{-2\nu} \int_{\R^d} ( 2\nu + \lambda^2 \norm[0]{\xi}^2 )^{\nu + d/2} \lvert \widehat{\phi}(\xi) \rvert^2 \dif \xi \\
      &\leq C_\nu B_\lambda q^{-2\nu} \int_{\R^d} ( 2\tau + \lambda^2 \norm[0]{\xi}^2 )^{\tau + d/2} \lvert \widehat{\phi}(\xi) \rvert^2 \dif \xi,
      \end{split}
  \end{equation*}
  where \smash{$B_\lambda = \max\{1, \lambda^{-(2\tau + d)}\}$}.
  By assumption~\eqref{eq:matern-factor-assumption-lower}, there is $c > 0$ such that \smash{$\Gamma(\nu + d/2) \geq c$} and $c(\nu) \geq c$ for all \smash{$\nu \in [\nu_\textup{min}, \tau]$}.
  Moreover, \smash{$\max_{ \nu > 0} (2\nu)^{-\nu} = e^{1/(2e)}$}.
  Thus
  \begin{equation*}
    \begin{split}
      C_\nu = \frac{\pi^{d/2}}{\sigma^2 c(\nu) 2^{\nu - 1} \Gamma(\nu + d/2)} \bigg( \frac{\lambda^2}{2\nu} \bigg)^\nu &\leq \frac{2 \pi^{d/2}}{\sigma^2 c^2 } \max\{1, \lambda^{2\tau}\} (2\nu)^{-\nu} \\
      &\leq \frac{\pi^{d/2}}{\sigma^2 c^2 } \max\{1, \lambda^{2\tau}\} \, e^{1/(2e)}
    \end{split}
  \end{equation*}
  for all $\nu \in [\nu_\textup{min}, \tau]$.
    It follows that there is a constant $C > 0$ such that $\norm[0]{g}_\nu \leq C q^{-\nu}$ for every $\nu \in [\nu_\textup{min}, \tau]$.
      Therefore the $H(K_\nu)$-norm of the function $f = C^{-1} q^\nu g$ does not exceed one.
      
      Set $q = q_{n}$ and assume that $n$ is sufficiently large that $q_n \leq 1$ holds.
      It follows from $\phi$ being supported on the unit ball and~\eqref{eq:gx-fun-def}, as well as the definition of the separation radius, that $f(x) = 0$ for every $x \in X_n^i$.
      By the argument given in the beginning of the proof and $q_n \leq 1$,
      \begin{equation*}
        \mathbb{V}_\nu(x_i \mid X_n^i) \geq \lvert f(x_i) \rvert^2 = C^{-2} q_n^{2\nu} \, g(x_i) = C^{-2} q_n^{2\nu} \phi(0) = C^{-2} q_n^{2\nu} \geq C^{-2} q_n^{2\tau}
      \end{equation*}
      for every $\nu \in [\nu_\textup{min}, \tau]$.
      The claim then follows from the assumption that $\{x_i\}_{i=1}^\infty$ are quasi-uniform and~\eqref{eq:quasi-uniform-rate}.
\end{proof}

\subsubsection{Some Norm Bounds}

We need upper and lower bounds on the RKHS norm of the conditional mean.
These are given in \Cref{prop:rkhs-norm-bound,prop:rkhs-norm-bound-lower}, respectively.

\begin{proposition} \label{prop:rkhs-norm-bound}
  Suppose that $\Omega \subset \R^d$ is bounded and that the points $\{x_i\}_{i=1}^\infty$ are quasi-uniform on $\Omega$.
  Let $B$ be a bounded subset of $H^{\tau + d/2}(\R^d)$ for $\tau > 0$.
  Then
  \begin{equation*}
    \sup_{f \in B} \norm[0]{ \mu_{\nu, f}(\cdot \mid X_n) }_{\nu}^2 \lesssim_n n^{2(\nu - \tau)/d}
  \end{equation*}
  for every $\nu \geq \tau$.
\end{proposition}
\begin{proof}
The claim follows from Lemma~A.1 in \citet{Karvonen2020}, quasi-uniformity, and the norm-equivalence of $H(K_{\nu})$ and $H^{\nu + d/2}(\R^d)$.
\end{proof}

The following lemma states that the Matérn norm for $\tau > 0$ is weaker, in a uniform sense, than that for all $\nu \geq \tau$ if the standard scaling $c(\nu) = 2^{1-\nu} / \Gamma(\nu)$ is used.
\rev{It seems likely that this lemma exists in some form in the literature.
Note that for Sobolev norms we obtain from~\eqref{eq:sobolev-norm} the simpler result that $\norm[0]{f}_{H^\beta(\R^d)} \leq \norm[0]{f}_{H^\alpha(\R^d)}$ if $\alpha \geq \beta$.}

\begin{lemma} \label{lemma:matern-norm-comparison}
\rev{Let $\tau > 0$ and suppose that there is a positive constant $C_\tau$ such that 
\begin{equation} \label{eq:matern-factor-assumption-norm-comparison}
  c(\nu) \leq \frac{C_\tau}{ 2^\nu \Gamma(\nu) }
\end{equation}
for all $\nu \geq \tau$.
Then there is a positive constant~$C$ such that
\begin{equation} \label{eq:matern-norm-comparison}
  \norm[0]{f}_{\tau}^2 \leq C \norm[0]{f}_{\nu}^2
\end{equation}
for all $\nu \geq \tau$ and $f \in H(K_\tau)$, where we set $\norm[0]{f}_{\nu} = \infty$ if $f \notin H(K_\nu)$.
In particular, the estimate~\eqref{eq:matern-norm-comparison} holds for all $\nu \in [\tau, \nu_\textup{max}]$ if $\tau \in [\nu_\textup{min}, \nu_\textup{max}]$ and $c(\nu)$ satisfies~\eqref{eq:matern-factor-assumption-upper}.}
\end{lemma}
\begin{proof}
It is straightforward to compute that the function
\begin{equation*}
  p_{\tau,\nu}(x) = \frac{(2\tau/\lambda^2 + x)^{\tau + d/2}}{(2\nu/\lambda^2 + x)^{\nu + d/2}}
\end{equation*}
attains its maximum on $[0, \infty)$ at $x = d / \lambda^2$.
Therefore
\begin{equation} \label{eq:p-func-upper-bound}
  \begin{split}
    \max_{x \geq 0} p_{\tau,\nu}(x) = \frac{(2\tau/\lambda^2 + d/\lambda^2)^{\tau + d/2}}{(2\nu/\lambda^2 + d/\lambda^2)^{\nu + d/2}} &= \bigg( \frac{\lambda^2}{2} \bigg)^{\nu-\tau}  \frac{\tau^{\tau + d/2}}{\nu^{\nu + d/2}} \cdot \frac{(1 + d/(2\tau))^{\tau + d/2}}{(1 + d/(2\nu))^{\nu + d/2}} \\
    &\leq \bigg( \frac{\lambda^2}{2} \bigg)^{\nu-\tau} \frac{\tau^{\tau + d/2}}{\nu^{\nu + d/2}} (1 + d/(2\tau))^{\tau + d/2}.
  \end{split}
\end{equation}
Let $\kappa_\tau = c(\tau) 2^{\tau-1} \Gamma(\tau + d/2)$ and $\lambda_\tau = (1 + d/(2\tau))^{\tau + d/2}$.
From~\eqref{eq:p-func-upper-bound} we get
\begin{equation*}
  \begin{split}
    \norm[0]{ f }_{\tau}^2 ={}& \frac{\pi^{d/2}}{\sigma^2 \kappa_\tau} \bigg( \frac{\lambda^2}{2\tau} \bigg)^{\tau} \int_{\R^d} \bigg( \frac{2\tau}{\lambda^2} + \norm[0]{\xi}^2 \bigg)^{\tau + d/2} \lvert \widehat{f}(\xi) \rvert^2 \dif \xi \\
    ={}& \frac{\pi^{d/2}}{\sigma^2 \kappa_\tau} \bigg( \frac{\lambda^2}{2\tau} \bigg)^{\tau} \int_{\R^d} p_{\tau,\nu}(\norm[0]{\xi}^2) \bigg( \frac{2\nu}{\lambda^2} + \norm[0]{\xi}^2 \bigg)^{\nu + d/2} \lvert \widehat{f}(\xi) \rvert^2 \dif \xi \\ 
    \leq{}& \frac{\pi^{d/2}}{\sigma^2 \kappa_\tau} \bigg( \frac{\lambda^2}{2\tau} \bigg)^{\tau} \bigg( \frac{\lambda^2}{2} \bigg)^{\nu-\tau} \frac{\tau^{\tau + d/2}}{\nu^{\nu + d/2}} \lambda_\tau \int_{\R^d} \bigg( \frac{2\nu}{\lambda^2} + \norm[0]{\xi}^2 \bigg)^{\nu + d/2} \lvert \widehat{f}(\xi) \rvert^2 \dif \xi \\
    ={}& \frac{\lambda_\tau \tau^{d/2}}{\kappa_\tau} \cdot \frac{\pi^{d/2}}{\sigma^2 \nu^{d/2}} \bigg( \frac{\lambda^2}{2\nu} \bigg)^{\nu}  \int_{\R^d} \bigg( \frac{2\nu}{\lambda^2} + \norm[0]{\xi}^2 \bigg)^{\nu + d/2} \lvert \widehat{f}(\xi) \rvert^2 \dif \xi \\
    ={}& \frac{\lambda_\tau \tau^{d/2}}{\kappa_\tau} \cdot \frac{c(\nu) 2^{\nu-1} \Gamma(\nu + d/2)}{\nu^{d/2}} \norm[0]{f}_\nu^2.
  \end{split}
\end{equation*}
The claim then follows from~\eqref{eq:matern-factor-assumption-norm-comparison} and the Gamma function asymptotics $\Gamma(\nu + d/2) \sim \nu^{d/2} \Gamma(\nu)$ as $\nu \to \infty$.
\end{proof}

For self-similar functions whose support is contained in the domain the RKHS norm of the conditional mean can be bounded from below by the following proposition, which originally appeared as Theorem~8 in \citet{VaartZanten2011}.

\begin{proposition} \label{prop:rkhs-norm-bound-lower}
Suppose that $\Omega \subset \R^d$ satisfies \Cref{ass:domain}, that $c(\nu)$ satisfies~\eqref{eq:matern-factor-assumption-upper}, and that the points $\{x_i\}_{i=1}^\infty$ are quasi-uniform on $\Omega$.
Let \smash{$\tau \in (0, \nu_\textup{max}]$}.
Then for every $\smash{f \in H_\textup{ss}^{\tau + d/2}(\Omega)}$ we have
\begin{equation} \label{eq:rkhs-norm-lower-bound-new}
  \inf_{\nu \in [\nu_1, \nu_2]} \norm[0]{\mu_{\nu, f}(\cdot \mid X_n) }_\nu^2 \gtrsim_n n^{2(\nu_1 - \tau)/d - 4(\nu_1 - \tau) \delta}
\end{equation}
for any $\delta \in (0, \tau)$, $\nu_1 \in [\tau, \nu_\textup{max}]$, and $\nu_2 \in [\nu_1, \nu_\textup{max}]$.
\end{proposition}
\begin{proof}
  By the minimum-norm property~\eqref{eq:mean-is-min-norm} and \Cref{lemma:matern-norm-comparison}, there is a positive constant $C$ such that
  \begin{equation*}
    \begin{split}
      \norm[0]{\mu_{\nu_1, f}(\cdot \mid X_n) }_{\nu_1}^2 \leq \norm[0]{\mu_{\nu, f}(\cdot \mid X_n) }_{\nu_1}^2 &\leq C \norm[0]{\mu_{\nu, f}(\cdot \mid X_n) }_{\nu}
      \end{split}
  \end{equation*}
  for all $\nu \in [\tau, \nu_\textup{max}]$.
  It therefore suffices to prove~\eqref{eq:rkhs-norm-lower-bound-new} for the fixed smoothness $\nu = \nu_1$ (i.e., when $\nu_2 = \nu_1$).
  Proposition~4.7 in \citet{Karvonen2020} with $\alpha = \nu_1 + d/2$, $\beta = \tau + d/2 \leq \alpha$ and $\gamma = \tau - \delta + d/2 < \beta$ states that 
  \begin{equation} \label{eq:karvonen-norm-lower-bound}
    \norm[0]{\mu_{\nu_1, f}(\cdot \mid X_n) }_{\nu_1}^2 \gtrsim_n n^{2\gamma(\alpha/\beta -1)/d} = n^{2(\nu_1 - \tau)/d - 2(\nu_1 - \tau)\delta/(d(\tau+d/2)) }
  \end{equation}
  if $f \in H^\gamma(\R^d)$ is such that $\lvert\widehat{f}(\xi)\rvert^2 \geq C (1 + \norm[0]{\xi}^2)^{-(\beta + d/2)}$ for some $C > 0$ and all $\xi \in \R^d$ and the support of $f$ is contained in $\Omega$.
  This assumption on the decay of the Fourier transform implies~\eqref{eq:self-similar-lower}.
  However, by inspection of the proofs of Proposition~4.7 in \citet{Karvonen2020} [specifically, the equation after~(A.10)] and Theorem~8 in \citet{VaartZanten2011} we see that~\eqref{eq:self-similar-lower} is sufficient to establish~\eqref{eq:karvonen-norm-lower-bound}.
  Moreover, that $\floor{\gamma} > d/2$ may be relaxed to $\gamma > d/2$ by using Corollary~4.1 of \citet{Arcangeli2007} to derive Equation~(A.9) in \citet{Karvonen2020}.
  For any $\smash{f \in H_\textup{ss}^{\tau + d/2}(\Omega)}$ we therefore have
  \begin{equation*}
    \norm[0]{\mu_{\nu_1, f}(\cdot \mid X_n) }_{\nu_1}^2 \gtrsim_n = n^{2(\nu_1 - \tau)/d - 2(\nu_1 - \tau)\delta/(d(\tau+d/2)) } \geq n^{2(\nu_1 - \tau)/d - 4(\nu_1 - \tau) \delta},
  \end{equation*}
  which completes the proof.
\end{proof}

\subsubsection{Main Result for Matérns} \label{sec:main-matern}

\Cref{prop:matern-upper,prop:matern-lower,prop:rkhs-norm-bound,prop:rkhs-norm-bound-lower} yield our main result on smoothness estimation in the Matérn model.
We shall use a parameter space of the following form:
\begin{equation} \label{eq:Theta-matern}
  \Theta = [\nu_\textup{min}, \nu_\textup{max}] \quad \text{ for } \quad \nu_\textup{min} \in (0, \nu_\textup{max}] \: \text{ and } \: \nu_\textup{max} \in (d/2, \infty).
\end{equation}

\begin{theorem} \label{thm:matern-main}
  Let $\Theta$ be given in~\eqref{eq:Theta-matern}.
  Suppose that $c(\nu)$ satisfies \eqref{eq:matern-factor-assumption-lower}, that $\Omega \subset \R^d$ satisfies \Cref{ass:domain}, and that the points $\{x_i\}_{i=1}^\infty$ are quasi-uniform on $\Omega$.
  If $B$ is a bounded subset of $H^{\nu_0}(\Omega)$ for \smash{$\nu_0 \in (d/2, \nu_\textup{max}]$}, then
  \begin{equation} \label{eq:matern-main-lower}
    \liminf_{n \to \infty} \inf_{f_0 \in B} \hat{\nu}_\ML^{f_0} (X_n) \geq \nu_0 \quad \text{ and } \quad \liminf_{n \to \infty} \inf_{f_0 \in B} \hat{\nu}_\CV^{f_0}(X_n) \geq \nu_0 - d/2.
  \end{equation}
  Moreover, the bound for maximum likelihood estimation is sharp if $c(\nu)$ satisfies also \eqref{eq:matern-factor-assumption-upper}, in the sense that for every $\varepsilon > 0$ there is $f_0 \in H^{\nu_0}(\Omega)$ such that
  \begin{equation} \label{eq:matern-main-upper}
    \limsup_{n \to \infty} \hat{\nu}_\ML^{f_0}(X_n) \leq \nu_0 + \varepsilon.
  \end{equation}
\end{theorem}
\begin{proof}
By~\eqref{eq:sobolev-space-restricted}, each $f_0 \in B \subset H^{\nu_0}(\Omega)$ may be identified with $f_e \in H^{\nu_0}(\R^d)$ such that $f_e|_\Omega = f_0$ and the Sobolev norms of $f_0$ and $f_e$ are equal.
The set of these extensions $f_e$ is bounded in $H^{\nu_0}(\R^d)$. 
We may therefore proceed as if $B$ were a bounded subset of $H^{\nu_0}(\R^d)$.
Let $\tilde{\nu} > 0$ and $\nu_1 \in \Theta$.
\Cref{prop:matern-upper,prop:matern-lower} yield
\begin{equation} \label{eq:matern-proof-var-ratio}
  \sup_{ \nu \in [\nu_\textup{min}, \nu_1] } \max_{1 \leq i \leq n} \frac{ \mathbb{V}_{\tilde{\nu}}(x_i \mid X_n^i) }{ \mathbb{V}_{\nu}(x_i \mid X_n^i) } \lesssim_n \frac{n^{-2\tilde{\nu}/ d}}{n^{-2\nu_1/d}} = n^{-2(\tilde{\nu} - \nu_1)/d}.
\end{equation}
Because $H(K_{\nu_0 - d/2})$ is norm-equivalent to $H^{\nu_0}(\R^d)$, the bound on the lower limit for $\hat{\nu}_\CV^{f_0}(X_n)$ follows by setting $\theta_0 = \tilde{\nu} = \nu_0 - d/2$ and $\theta_1 = \nu_1 < \nu_0 - d/2 < \tilde{\nu}$ in \Cref{cor:general-interval}, in which case $n^{-2(\tilde{\nu} - \nu_1)/d} \to 0$.
Note that the lower limit holds trivially whenever $\nu_0 - d/2 < \nu_\textup{min}$.

To prove the lower bound for the maximum likelihood estimator we shall apply \Cref{cor:general-interval-new} with $\theta_0 = \nu_0$.
If $\nu_0 < \nu_\textup{min}$, the lower bound holds trivially, so we may assume that $\nu_0 \in \Theta$.
Setting $\tilde{\nu} = \nu_0$ and considering the case $i = n$ (so that $X_n^i = X_{n-1}$) in~\eqref{eq:matern-proof-var-ratio} yields
\begin{equation} \label{eq:matern-proof-var-ratio-2}
  \sup_{ \nu \in [\nu_\textup{min}, \nu_1] } \frac{ \mathbb{V}_{\nu_0}(x_n \mid X_{n-1}) }{ \mathbb{V}_{\nu}(x_n \mid X_{n-1}) } \lesssim_n n^{-2(\nu_0 - \nu_1)/d}.
\end{equation}
\Cref{prop:rkhs-norm-bound} with $\nu = \nu_0$ and $\tau = \nu_0 - d/2$ yields
\begin{equation*}
  \sup_{f_0 \in B}\norm[0]{ \mu_{\nu_0, f_0}(\cdot \mid X_n) }_{\nu_0}^2 \lesssim_n n^{2(\nu - \tau)/d} = n.
\end{equation*}
Therefore,
\begin{equation*}
  \begin{split}
    Q_n(\nu_1) &\coloneqq \sup_{ \nu  \in [\nu_\textup{min}, \nu_1] } \sup_{f_0 \in B} \Bigg[ \norm[0]{\mu_{\nu_0, f_0}(\cdot \mid X_n)}_{\nu_0}^2 + \sum_{i=1}^n \log \frac{\mathbb{V}_{\nu_0}(x_{i} \mid X_{i-1})}{\mathbb{V}_{\nu}(x_{i} \mid X_{i-1})} \Bigg] \\
      &\lesssim_n n - \frac{2(\nu_0 - \nu_1)}{d} \sum_{i=1}^n \log i \\
      &= n - \frac{2(\nu_0 - \nu_1)}{d} \log n!.
  \end{split}
\end{equation*}
From Stirling's formula $n! \sim \sqrt{2\pi} n^{n+1/2} e^{-n}$ as $n \to \infty$ we get
\begin{equation*}
  \frac{2(\nu_0 - \nu_1)}{d} \log n! \sim \frac{2(\nu_0 - \nu_1)}{d} n \log n.
\end{equation*}
Hence
\begin{equation*}
  Q_n(\nu_1) \lesssim_n n - \frac{2(\nu_0 - \nu_1)}{d} \log n! \to -\infty
\end{equation*}
as $n \to \infty$ for every $\nu_1$ such that $2(\nu_0 - \nu_1)/d > 0$, which is equivalent to $\nu_1 < \nu_0$.
We may thus use \Cref{cor:general-interval-new} with $\theta_0 = \nu_0$.

We are left to prove that the lower bound for the maximum likelihood estimator is sharp in the sense that for every $\varepsilon > 0$ there is $f_0 \in H^{\nu_0}(\Omega)$ for which~\eqref{eq:matern-main-upper} holds.
We may assume that $\nu_0 + \varepsilon < \nu_\textup{max}$, for otherwise~\eqref{eq:matern-main-upper} holds trivially.
We shall use \Cref{cor:upper-bounds-interval}.
By \Cref{lemma:self-similar-inclusion,lemma:ft-lower-bound} and \Cref{prop:rkhs-norm-bound-lower} with $\tau = \nu_0 + \varepsilon/2 - d/2$, $\nu_1 = \nu_0 + \varepsilon$, and $\delta$ sufficiently small, there exists $\smash{f_0 \in H^{\nu_0}(\R^d)}$ such that 
\begin{equation} \label{eq:norm-lower-bound-in-main-proof}
  \begin{split}
  \inf_{\nu \in [\nu_1, \nu_2]} \norm[0]{\mu_{\nu, f_0}(\cdot \mid X_n) }_\nu^2 \gtrsim_n n^{2(\nu_1 - \tau)/d - 4(\nu_1 - \tau) \delta} &= n^{1 + \varepsilon/d - 4(d/2 + \varepsilon/2)\delta} \\
  &\geq n^{1 + \varepsilon/(2d)}
  \end{split}
\end{equation}
for every $\nu_2 \in [\nu_1, \nu_\textup{max}]$.
The lower bound~\eqref{eq:norm-lower-bound-in-main-proof} and \Cref{prop:matern-lower} yield
\begin{equation} \label{eq:P-matern-proof}
  \begin{split}
    P_n(\nu_1, \nu_2) &\coloneqq \inf_{\nu \in [\nu_1, \nu_2]} \ell_\ML^{f_0}( \nu \mid X_n ) \\
    &= \inf_{\nu \in [\nu_1, \nu_2]} \Bigg[ \norm[0]{\mu_{\nu, f_0}(\cdot \mid X_n)}_{\nu}^2 + \sum_{i=1}^n \log \mathbb{V}_{\nu}(x_{i} \mid X_{i-1}) \Bigg] \\
    &\gtrsim_n n^{1 + \varepsilon/(2d)} - \frac{2\nu_2}{d} \log n!,
  \end{split}
\end{equation}
so that $P_n(\nu_1, \nu_2) \to \infty$ as $n \to \infty$ by the same arguments that we used to control $Q_n(\nu_1)$ above.
The claim follows by applying \Cref{cor:upper-bounds-interval} with $\theta_0 = \nu_1 = \nu_0 + \varepsilon$ and $B = \{f_0\}$ for every $\theta_1 = \nu_2 > \nu_1$.
\end{proof}

As discussed in \Cref{sec:parameter-estimation}, \Cref{cor:general-interval} that we applied to cross-validation is weaker than \Cref{cor:general-interval-new} that we applied to maximum likelihood estimation. \Cref{cor:general-interval} is capable of establishing a lower bound $\nu_0$ only if $B$ is a bounded subset of $H^{\nu_0 + d/2}(\Omega)$ while \Cref{cor:general-interval-new} is more flexible. 
The use of \Cref{cor:general-interval-new} to improve the lower bound for cross-validation would require that one proved a variant of \Cref{prop:rkhs-norm-bound} to bound
\begin{equation*}
  \sum_{i=1}^n \frac{(f_0(x_{i}) - \mu_{\nu, f_0}(x_i \mid X_n^i))^2}{\mathbb{V}_{\nu}( x_i \mid X_n^i )}
\end{equation*}
from above when $\nu \geq \nu_0$. 
We do not presently know how to do this.
Similarly, proving a version of the upper bound~\eqref{eq:matern-main-upper} for cross-validation would require controlling the above quantity from below.

The following theorem optimises \Cref{thm:matern-main} for the maximum likelihood estimator.
For a function $f$ that is defined and square-integrable at least on~$\Omega$, let
\begin{equation*} \label{eq:smoothness-f}
  \nu(f) = \sup \Set{\nu > 0}{f|_\Omega \in H^{\nu}(\Omega)},
\end{equation*}
so that $f \in H^{\nu(f) - \varepsilon}(\Omega)$ for every $\varepsilon \in (0, \nu(f))$.
\citet[Section~5.1.1]{WangJing2021} call $\nu(f)$ the \emph{smoothness of} $f$.

\begin{theorem} \label{thm:matern-main-2}
  Let $\Theta$ be given in~\eqref{eq:Theta-matern}.
  Suppose that $c(\nu)$ satisfies \eqref{eq:matern-factor-assumption-lower}, that $\Omega \subset \R^d$ satisfies \Cref{ass:domain}, and that the points $\{x_i\}_{i=1}^\infty$ are quasi-uniform on $\Omega$.
  \begin{enumerate}
  \item If $f_0 \colon \Omega \to \R$ is such that $\nu(f_0) \in (d/2, \nu_\textup{max}]$, then
  \begin{equation} \label{eq:matern-ml-lower-bound-2}
    \liminf_{n \to \infty} \hat{\nu}_\ML^{f_0} (X_n) \geq \nu(f_0).
  \end{equation}
  \item If $f_0 \in H_\textup{ss}^{\nu_0}(\Omega)$ for $\nu_0 \in (d/2, \nu_\textup{max}]$ and $c(\nu)$ satisfies also \eqref{eq:matern-factor-assumption-upper}, then
  \begin{equation} \label{eq:matern-ml-exact-limit}
    \lim_{n \to \infty} \hat{\nu}_\ML^{f_0}(X_n) = \nu_0.
  \end{equation}
  \end{enumerate}
\end{theorem}
\begin{proof}
  The bound~\eqref{eq:matern-ml-lower-bound-2} follows from~\eqref{eq:matern-main-lower} because $B = \smash{\{f_0\} \subset H^{\nu(f_0) - \varepsilon}(\Omega)}$ for every sufficiently small $\varepsilon > 0$.
  Suppose that $f_0 \in \smash{H_\textup{ss}^{\nu_0}(\Omega)}$ and $\nu_0 < \nu_\textup{max}$.
  Let $\varepsilon > 0$ be small enough that $\nu_0 + \varepsilon < \nu_\textup{max}$.
  With $\tau = \nu_0 - d/2$, $\nu_1 = \nu_0 + \varepsilon$, and $\delta$ sufficiently small, \Cref{prop:rkhs-norm-bound-lower} yields
\begin{equation*}
  \inf_{\nu \in [\nu_1, \nu_2]} \norm[0]{\mu_{\nu, f_0}(\cdot \mid X_n) }_\nu^2 \gtrsim_n n^{2(\nu_1 - \tau)/d - 4(\nu_1 - \tau) \delta} = n^{1 + 2\varepsilon/d - 4(d/2+\varepsilon)\delta} \geq n^{1 + \varepsilon/d}
\end{equation*}
for every $\nu_2 \in [\nu_1, \nu_\textup{max}]$.
By bounding $P_n(\nu_1, \nu_2) = \inf_{\nu \in [\nu_1, \nu_2]} \ell_\ML^{f_0}( \nu \mid X_n )$ as in the proof of \Cref{thm:matern-main} and applying \Cref{cor:upper-bounds-interval} with $\theta_0 = \nu_0$ and $B = \{f_0\}$ to the parameter space $\Theta = [\nu_\textup{min}, \nu_\textup{max}]$, we then get $\smash{\limsup_{n \to \infty} \hat{\nu}_\ML^{f_0}(X_n)} \leq \nu_0$.
If $\nu_0 = \nu_\textup{max}$, this upper bound holds trivially.
To establish a matching lower bound, observe that, by \Cref{lemma:self-similar-inclusion}, $\smash{f_0|_\Omega \in H^{\nu}(\Omega)}$ for every $\nu < \nu_0$.
Therefore $\nu(f_0) \geq \nu_0$, and it thus follows from~\eqref{eq:matern-ml-lower-bound-2} that \smash{$\liminf_{n \to \infty} \hat{\nu}_\ML^{f_0} (X_n) \geq \nu_0$}.
This concludes the proof of~\eqref{eq:matern-ml-exact-limit}.
\end{proof}

\rev{Let us say a few words about the assumptions in \Cref{thm:matern-main,thm:matern-main-2}.
As discussed earlier, \Cref{ass:domain} on the domain is non-restrictive.
In comparison to other spatial sampling assumptions employed in the literature, the quasi-uniformity assumption is rather non-restrictive and generic.
However, random points are not quasi-uniform.
Significant relaxations of the quasi-uniformity assumption may require a new assumption on the spatial homogeneity of $f_0$.
The assumption that the parameter space is compact and bounded away from zero is not practically restrictive.
Generalisations to the case $\Theta = (0, \infty)$ would require much more careful handling of smoothness-dependent constant coefficients.
For example, in the proof of~\eqref{eq:matern-main-upper} the two terms that make up $P_n(\nu_1, \nu_2)$ in~\eqref{eq:P-matern-proof} are bounded from below individually, which allows us to ignore potential dependencies on $\nu_2$ in their constant coefficients.

That the self-similar functions for which the limit~\eqref{eq:matern-ml-exact-limit} is obtained are required to have their supports in $\Omega$ is the most unsatisfactory part of \Cref{thm:matern-main,thm:matern-main-2}.
Because the samples from a Gaussian process with a Matérn covariance kernel are not compactly supported, this assumption is likely superfluous.
The assumption propagates from \Cref{prop:rkhs-norm-bound-lower} and Theorem~8 in \citet{VaartZanten2011} and to remove it a new technique to obtain lower bounds on RKHS norms is needed.}

\subsection{Infinitely Smooth Functions}

The following corollary shows that the parameter estimators detect infinite smoothness of the response function.
Here we obviously need to consider an infinite parameter space.

\begin{corollary} \label{cor:matern-infinitely-smooth}
  Let $\Theta = [\nu_\textup{min}, \infty)$ for $\nu_\textup{min} > 0$.
  Suppose that $c(\nu)$ satisfies~\eqref{eq:matern-factor-assumption-lower} for every $\nu_\textup{max} > \nu_\textup{min}$, that $\Omega \subset \R^d$ satisfies \Cref{ass:domain}, and that the points $\{x_i\}_{i=1}^\infty$ are quasi-uniform on $\Omega$.
  If $B$ is a bounded subset of $H^{\nu}(\Omega)$ for every $\nu > d/2$, then
  \begin{equation*}
    \lim_{n \to \infty} \inf_{f_0 \in B} \hat{\nu}_\ML^{f_0}(X_n) = \lim_{n \to \infty} \inf_{f_0 \in B} \hat{\nu}_\CV^{f_0}(X_n) = \infty.
  \end{equation*}
\end{corollary}
\begin{proof}
  The claims follow because the bounds~\eqref{eq:matern-main-lower} hold for every $\nu_0 > d/2$.
\end{proof}

Even though the Matérn kernel, with proper parametrisation, tends to the infinitely smooth Gaussian kernel as $\nu \to \infty$, it is not required in \Cref{cor:matern-infinitely-smooth} that $f_0$ be an element of the RKHS of the Gaussian kernel or of an RKHS that contains the samples of a Gaussian process with the Gaussian kernel.
However, as we now argue, it is likely that membership in the Gaussian RKHS causes the smoothness parameter estimates to diverge faster as $n$ increases than if this were not the case.
For simplicity, consider \Cref{thm:general} and maximum likelihood estimation.
From the proof of \Cref{thm:general} we see that $\hat{\nu}_\ML^{f_0}(X_n) \geq \nu_1$ if $f \in H(K_{\nu_0})$ and 
\begin{equation*}
  \norm[0]{f_0}_{\nu_0}^2 + \sup_{ \nu \in [\nu_\textup{min}, \nu_1]} \sum_{i=1}^{n} \log \frac{ \mathbb{V}_{\nu_0}(x_{i} \mid X_{i-1}) }{ \mathbb{V}_{\nu}(x_{i} \mid X_{i-1}) } < 0
\end{equation*}
for $\nu_1 < \nu_0$.
Notably, the second term above does not depend on the response function.
This suggests (though does not rigorously prove) that the larger $\norm[0]{f_0}_{\nu_0}^2$ is, the larger an $n$ is needed for the maximum likelihood estimator to exceed $\nu_1$.
Or, in other words, a fast rate of growth of $\norm[0]{f_0}_{\nu}$ as $\nu$ increases ought to imply a slow rate of growth of $\hat{\nu}_\ML^{f_0}(X_n)$.
Next we estimate $\norm[0]{f_0}_{\nu}$ for two functions, of which one is not an element of the RKHS of the Gaussian kernel and the other is.

Consider the Matérn kernel
\begin{equation*}
  K_\nu(x, y) = \frac{2^{1-\nu}}{\Gamma(\nu)} \big( \sqrt{\nu} \abs[0]{x-y} \big)^\nu \mathcal{K}_\nu \big( \sqrt{\nu} \abs[0]{x-y} \big)
\end{equation*}
for $d = 1$, where we have set $\sigma = 1$ and $\lambda = \sqrt{2}$.
As $\nu \to \infty$, the Matérn kernel above tends to the Gaussian kernel $K(x, y) = \exp(-(x-y)^2/4)$.
By the characterisation~\eqref{eq:rkhs-fourier-transform}, the RKHS of the Gaussian kernel consists of functions whose Fourier transforms satisfy
\begin{equation} \label{eq:gauss-rkhs-infinitely-smooth}
  \int_\R \exp(\xi^2) \lvert \widehat{f}(\xi) \rvert^2 \dif \xi < \infty.
\end{equation}
Consider the function $f_0(x) = 1/(1/4+x^2)$ with the Fourier transform $\widehat{f}_0(\xi) = \exp(-\abs[0]{\xi} / 2)$.
It is clear from~\eqref{eq:matern-norm} and~\eqref{eq:gauss-rkhs-infinitely-smooth} that this function is an element of every Matérn RKHS but not an element of the RKHS of the Gaussian kernel.
Now, from~\eqref{eq:matern-norm} we have
\begin{equation*}
  \begin{split}
    \norm[0]{f_0}_\nu^2 = C_\nu \int_\R ( \nu + \xi^2 )^{\nu + 1/2} \exp ( -\abs[0]{\xi} ) \dif \xi
    &= 2 C_\nu \int_0^\infty ( \nu + \xi^2 )^{\nu + 1/2} \exp ( - \xi ) \dif \xi \\
    &\geq 2 C_\nu \int_0^\infty \xi^{2\nu+1} \exp ( - \xi ) \dif \xi \\
    &= 2 \sqrt{\pi} \, \frac{ \Gamma(\nu) \Gamma(2\nu + 2)}{\Gamma(\nu + 1/2)} \nu^{-\nu}.
    \end{split}
\end{equation*}
From $\Gamma(\nu + 1/2) \sim \Gamma(\nu) \sqrt{\nu}$ and Stirling's formula we get
\begin{equation*}
  \frac{ \Gamma(\nu) \Gamma(2\nu + 2)}{\Gamma(\nu + 1/2)} \nu^{-\nu} \sim \Gamma(2\nu + 2) \nu^{-\nu-1/2} \sim \sqrt{2\pi} \, \frac{2^{\nu+1/2}}{e^{2\nu+1}} (2\nu + 1)^{\nu+1},
\end{equation*}
which shows that $\smash{\norm[0]{f_0}_\nu \to \infty}$ very fast as $\nu \to \infty$.
We therefore expect that for this response function $\smash{\hat{\nu}_\ML^{f_0}(X_n)}$ grows rather slowly.

As a second example, consider the function \smash{$f_0(x) = (2 \sqrt{\pi}\,)^{-1} \exp(-x^2/4)$} with the Fourier transform \smash{$\widehat{f}_0(\xi) = \exp(-\xi^2)$}.
Unlike the previous function, this function is an element of the Gaussian RKHS because
\begin{equation*}
  \int_\R \exp(\xi^2) \lvert \widehat{f}_0(\xi) \rvert^2 \dif \xi = \int_\R \exp(-\xi^2) \dif \xi = \sqrt{\pi} < \infty.
\end{equation*}
Furthermore,
\begin{equation*}
  \begin{split}
    \norm[0]{f_0}_\nu^2 = C_\nu \int_\R (\nu + \xi^2)^{\nu + 1/2} \exp(-2\xi^2) \dif \xi &= C_\nu \nu^{\nu+1/2} \int_\R \bigg(1 + \frac{\xi^2}{\nu} \bigg)^{\nu + 1/2} \exp(-2\xi^2) \dif \xi \\
    &\leq C_\nu \nu^{\nu + 1/2} \int_\R \exp(\xi^2) \sqrt{ 1 + \frac{\xi^2}{\nu} } \exp(-2\xi^2) \dif \xi \\
    &= \sqrt{\pi} \, \frac{ \Gamma(\nu) \sqrt{\nu} }{\Gamma(\nu + 1/2)} \int_\R \sqrt{ 1 + \frac{\xi^2}{\nu} } \exp(-\xi^2) \dif \xi \\
    &\sim \sqrt{\pi} \int_\R \exp(-\xi^2) \dif \xi \\
    &= \pi.
    \end{split}
\end{equation*}
Therefore $\norm[0]{f_0}_\nu$ is bounded as $\nu \to \infty$ and we can thus expect that $\hat{\nu}_\ML^{f_0}(X_n)$ grows fast as $n$ increases.

\subsection{Sample Paths} \label{sec:sample-paths}

Let us then assume that $f_0$ is a version of a zero-mean Gaussian process with a Matérn covariance kernel $K_{\nu_0}$ that has continuous sample paths (such a version always exists).
It is well known that under these assumptions almost all samples of $f_0$ are in the Sobolev space of order $\nu_0 - \varepsilon$ for every $\varepsilon > 0$.
Specifically, if $\Omega \subset \R^d$ satisfies \Cref{ass:domain}, then $f_0 \in H^{\nu_0 - \varepsilon}(\Omega)$ almost surely for every $\varepsilon > 0$.
  For results of this type, see \citet{Scheuerer2010}; Corollary~4.15 in \citet{Kanagawa2018}; Corollaries~4.5 and~5.7 in \citet{Steinwart2019}; \citet{Karvonen2021}; and \citet{Henderson2022}.
  From this result we obtain a Bayesian version of \Cref{thm:matern-main}.

\begin{corollary} \label{cor:matern-samples}
  Let $\Theta$ be given in~\eqref{eq:Theta-matern}.
  Suppose that $c(\nu)$ satisfies \eqref{eq:matern-factor-assumption-lower}, that $\Omega \subset \R^d$ satisfies \Cref{ass:domain}, and that the points $\{x_i\}_{i=1}^\infty$ are quasi-uniform on $\Omega$.  
    If $f_0$ is a version of a zero-mean Gaussian process with a Matérn covariance kernel $K_{\nu_0}$ for some \smash{$\nu_0 \in (d/2, \nu_\textup{max}]$} such that almost all of its samples are continuous, then
  \begin{equation*}
    \liminf_{ n \to \infty } \hat{\nu}_\ML^{f_0}(X_n) \geq \nu_0 \quad \text{ and } \quad \liminf_{ n \to \infty } \hat{\nu}_\CV^{f_0}(X_n) \geq \nu_0 - \frac{d}{2} \quad \text{almost surely}.
  \end{equation*}
\end{corollary}
\begin{proof}
  Because $f_0 \in H^{\nu_0 - 1/k}(\Omega)$ almost surely for every $k \in \N$, we obtain from \Cref{thm:matern-main} that
  \begin{equation*} 
    \liminf_{ n \to \infty } \hat{\nu}_\ML^{f_0}(X_n) \geq \nu_0 - \frac{1}{k} \quad \text{almost surely}
  \end{equation*}
  for every $k \in \N$.
  The claims follow by taking intersections over $k \in \N$ of sets of measure one.
  The proof for cross-validation is analogous.
\end{proof}

As discussed in \Cref{sec:main-matern}, the result in \Cref{cor:matern-samples} for cross-validation is likely sub-optimal.
Because the samples are not supported on $\Omega$, we are unable to exploit results that require $f_0$ to be an element of $\smash{H_\textup{ss}^{\nu_0}(\Omega)}$.
Nevertheless, \Cref{cor:matern-samples} is a step towards showing that the maximum likelihood and cross-validation estimators are consistent or strongly consistent.
\citet[Theorem~2.7]{Chen2021} and \citet{Petit2022} have proved consistency results for smoothness estimators of periodic Matérn-type kernels.

\section{Discussion} \label{sec:discussion}

This section contains some discussion on the implications of the smoothness estimation results in \Cref{sec:matern}.
We also discuss the use of \Cref{thm:general} to estimate the scale parameter of an infinitely smooth stationary kernel.

\subsection{Approximation and Uncertainty Quantification} \label{sec:uq}

Suppose that $f_0 \in H^{\nu_0 + d/2}(\Omega)$ and that the assumptions of \Cref{thm:matern-main} are satisfied.
We now discuss what happens in the frequentist setting when a Matérn model of smoothness $\nu$ either over- or undersmooths the truth.
For a succinct review on the closely related topic of coverage properties of Bayesian credible intervals we refer to \citet[pp.\@~1391--2]{Szabo2015}.

\vspace{0.1cm}

\noindent \emph{Undersmoothing}: $\nu_0 > \nu$.
If the response function is smoother than the Matérn prior, the current theory does not guarantee that the conditional mean tends to the response function with a rate that adapts to the smoothness, except if (essentially) $\nu_0 \geq 2\nu$, in which case one can expect the rate to be approximately $n^{-2\nu/d}$.
This is known as \emph{superconvergence}~\citep{Schaback2018} or the \emph{improved rate}~\citep[Section~11.5]{Wendland2005} of kernel-based approximation.
See \citet[Sections~3.4 and~4.5]{Karvonen2020} and \citet[Section~2.3]{TuoWangWu2020} for discussion on superconvergence in the context of Gaussian process interpolation.
By~\eqref{eq:rkhs-error} and the fact that $f_0 \in H^{\nu_0 + d/2}(\Omega) \subset H^{\nu + d/2}(\Omega)$, 
\begin{equation} \label{eq:discussion-undersmoothing}
  \abs[0]{ f_0(x) - \mu_{\nu, f_0}(x \mid X_n) } \leq \norm[0]{f_0}_\nu \mathbb{V}_\nu(x \mid X_n)^{1/2}
\end{equation}
for every $x \in \Omega$.
Therefore $f_0$ is contained in the credible set
\begin{equation} \label{eq:discussion-credible-set}
  \mathcal{C}^n_\rho(f_0) = \Set[\big]{f \colon \Omega \to \R}{\abs[0]{ f(x) - \mu_{\nu, f_0}(x \mid X_n) } \leq \rho \mathbb{V}_\nu(x \mid X_n)^{1/2} \text{ for every } x \in \Omega }
\end{equation}
for any $\rho \geq \norm[0]{f_0}_\nu$ and $n \in \N$.
Reliable uncertainty quantification is therefore possible when the model undersmooths, in that there is a fixed credible level for which the credible set centered at the conditional mean contains the truth for all $n$.
If $f_0$ is smooth enough to benefit from superconvergence (or other such phenomenon), then $f_0 \in \mathcal{C}_{\rho_n}^n(f_0)$ can hold even for a sequence $(\rho_n)_{n=1}^\infty$ that decays fast.
Since $\smash{\| \mu_{\nu, f_0}( \cdot \mid X_n) \|_\nu^2 \to \norm[0]{f_0}_{\nu}}$ as $n \to \infty$ if $h_{n, \Omega} \to 0$~\citep[e.g.,][Theorem~8.37]{Iske2018}, one may use~\eqref{eq:mu-norm-explicit} to construct a sequence $(\bar{\rho}_n)_{n=1}^\infty$ with limit $\bar{\rho} = \norm[0]{f_0}_\nu = \lim_{n \to \infty} \bar{\rho}_n$ such that $f_0 \in \mathcal{C}_{\bar{\rho}}^n(f_0)$ for every $n$.

\vspace{0.1cm}

\noindent \emph{Oversmoothing}: $\nu_0 < \nu$.
By the Narcowich--Ward--Wendland escape theorem (\citealt{Narcowich2006}; see \citealt[Theorem~1]{Wynne2021} for a Gaussian process formulation), the Matérn conditional mean tends to $f_0 \in H^{\nu_0 + d/2}(\Omega)$ with a rate which depends on $\nu_0$:
\begin{equation} \label{eq:misspecified-rate}
  \sup_{x \in \Omega} \abs[0]{ f_0(x) - \mu_{\nu, f_0}(x \mid X_n) } \lesssim_n n^{-\nu_0 / d}.
\end{equation}
The rate in~\eqref{eq:misspecified-rate} is worst-case optimal because from~\Cref{prop:matern-upper,prop:matern-lower} and~\eqref{eq:var-is-wce} it follows that the worst-case error in $H^{\nu_0 + d/2}(\Omega)$ decays with the same $n^{-\nu_0 / d}$ rate.
This means that oversmoothing does not have an adverse effect on the rate of convergence of the conditional mean.
However, the model may provide overconfident uncertainty quantification because, by a combination of \Cref{prop:matern-upper,prop:matern-lower}, the conditional standard deviation decays with a rate $n^{-\nu / d}$, which is (potentially) faster than the rate on the right-hand side in~\eqref{eq:misspecified-rate}.
Therefore it is possible that
\begin{equation*}
  \limsup_{n \to \infty} \frac{ \abs[0]{ f_0(x) - \mu_{\nu, f_0}(x \mid X_n) } }{ \mathbb{V}_\nu(x \mid X_n)^{1/2}} = \infty
\end{equation*}
for some $x \in \Omega$, which implies there is no fixed $\rho$ for which $f_0$ is contained in the credible set in~\eqref{eq:discussion-credible-set} for all $n \in \N$.

\vspace{0.1cm}

\Cref{thm:matern-main} ensures that maximum likelihood estimation cannot can yield a Gaussian process model which is asymptotically undersmoothing.
The conditional mean therefore tends to the response function with a worst-case optimal rate.
However, there is much statistical literature on the impossibility of constructing, in the presence of noise, non-parametric confidence or credible bands (or sets) that are \emph{adaptive} over natural smoothness classes, such as Sobolev or Hölder spaces, in the sense that (i) the band contains the truth with high probability and (ii) the radius of the band decays with the worst-case optimal rate for any truth in one of the smoothness classes~\citep{Low1997, CaiLow2004, RobinsVaart2006}.
While these results do not directly apply to our setting, they do suggest that it is likely that more restrictive assumptions on $f_0$ than membership in Sobolev spaces are necessary if it is to be guaranteed that uncertainty quantification is not overconfident.
In the context of the white noise and the Gaussian sequence space model, it is known that adaptivity can be guaranteed over self-similar and related classes that exclude ``inconvenient'' or ``deceptive'' response functions whose smoothness cannot be estimated~\citep{Bull2012, Szabo2015, NicklSzabo2016}. 
\Cref{thm:matern-main-2} and these connections suggest that it is the class of self-similar Sobolev functions (or a closely related class) for which one should attempt to establish non-overconfidence of uncertainty quantification.

\subsection{Scale Estimation for Infinitely Smooth Kernels} \label{sec:smooth}

Let $\lambda > 0$ be a scale parameter and consider the infinitely smooth Gaussian kernel
\begin{equation*}
  K_\lambda(x, y) = \Phi_\lambda(x - y) = \Phi\bigg( \frac{x - y}{\lambda} \bigg), \quad \text{ where } \quad \Phi(z) = \exp\bigg( \! -\frac{1}{2} \norm[0]{z}^2 \bigg).
\end{equation*}
\rev{Rescaled Gaussian processes with Gaussian covariance kernel are studied in \citet{VaartZanten2009}}, while results on scale parameter estimation for a kernel related to the Gaussian can be found in \citet{HadjiSzabo2021}.
The Fourier transform of $\Phi$ is
\begin{equation*}
  \widehat{\Phi}(\xi) = (2\pi)^{d/2} \exp\bigg(\!-\frac{1}{2} \norm[0]{\xi}^2 \bigg),
\end{equation*}
and from the scaling properties of the Fourier transform we get
\begin{equation*} 
  \widehat{\Phi}_\lambda(\xi) = \lambda^d \, \widehat{\Phi}( \lambda \xi) = (2\pi \lambda^2)^{d/2} \exp\bigg(\!-\frac{1}{2} \lambda^2 \norm[0]{\xi}^2 \bigg).
\end{equation*}
Recall from \Cref{sec:sobolev} that $H(K_\lambda)$ contains those square-integrable and continuous functions $f \colon \R^d \to \R$ for vwhich
\begin{equation} \label{eq:gauss-norm}
  \norm[0]{f}_\lambda^2 = \int_{\R^d} \frac{\abs[0]{\widehat{f}(\xi)}^2}{\widehat{\Phi}_\lambda(\xi)} \dif \xi = \frac{1}{(2\pi\lambda^2)^{d/2}} \int_{\R^d} \exp\bigg(\frac{1}{2} \lambda^2 \norm[0]{\xi}^2 \bigg) \lvert \widehat{f}(\xi) \rvert^2 \dif \xi < \infty.
\end{equation}
From~\eqref{eq:gauss-norm} we see that $H(K_{\lambda_0})$ is a proper subset of $H(K_{\lambda_1})$ whenever $\lambda_0 > \lambda_1$, which is a well-known result.
The reasoning in \Cref{sec:parameter-estimation} therefore suggests that~\eqref{eq:variance-decay-interval-ml} and~\eqref{eq:variance-decay-interval-cv} may hold if $f_0 \in H(K_{\lambda_0})$ for some $\lambda_0 > 0$ and, consequently,
\begin{equation*}
  \liminf_{n \to \infty} \hat{\lambda}^{f_0}_\ML(X_n) \geq \lambda_0 \quad \text{ and } \quad \liminf_{n \to \infty} \hat{\lambda}^{f_0}_\CV(X_n) \geq \lambda_0.
\end{equation*}
Similar reasoning applies to many other infinitely smooth stationary kernels, such as the Cauchy kernel defined by $\Phi(z) = (1 + \norm[0]{z}^2)^{-1}$, whose Fourier transforms decay with super-algebraic rates.

However, to prove that either of the assumptions~\eqref{eq:variance-decay-interval-ml} or~\eqref{eq:variance-decay-interval-cv} holds for the Gaussian kernel does not appear to be possible at the moment.
Upper bounds on the conditional variance of the form
\begin{equation*}
  \sup_{x \in \Omega} \mathbb{V}_\lambda( x \mid X_n ) \lesssim_n \exp(-c_\lambda n^{1/d} \log n),
\end{equation*}
where and $c_\lambda$ is a positive constant which depends on $\lambda$, are available in the scattered data approximation literature~\citep[e.g.,][Theorem~6.1]{RiegerZwicknagl2010} if $\Omega \subset \R^d$ is bounded and sufficiently regular and the points $\{x_i\}_{i=1}^\infty$ are such that $h_{n,\Omega} \lesssim_n n^{-1/d}$, but no corresponding lower bounds exist.
Moreover, the constant $c_\lambda$ in these results is unlikely to be optimal.
The result that is closest to being useful is a theorem in \citet{Karvonen2022-power-series}, which states that for the univariate Gaussian kernel we have
\begin{equation*}
  C_{1,\lambda} \frac{1}{(4\lambda^2)^n n!} \leq \sup_{ x \in [-1, 1] } \mathbb{V}_\lambda( x \mid X_n ) \leq C_{2,\lambda} \, n^{-1/4} e^{2\sqrt{n}/\lambda} \frac{1}{(4\lambda^2)^n n!}
\end{equation*}
for sufficiently large $n$ and positive constants $C_{1,\lambda}$ and $C_{2,\lambda}$ that depend on $\lambda$ if $X_n = \{x_{n,i}\}_{i=1}^n$ are the non-nested Chebyshev nodes $x_{n,i} = \cos(\pi(i-1/2)/n)$.
\rev{See also Theorem~2 in \citet{Yarotsky2013}.}
In this setting the rate of convergence is controlled by $\lambda$ because
\begin{equation} \label{eq:gaussian-kernel-sup-ratio}
  n^{1/4} e^{-2\sqrt{n} / \lambda} \bigg( \frac{\lambda^2}{\lambda_0^2} \bigg)^n \lesssim_n \frac{\sup_{x \in [-1, 1]}\mathbb{V}_{\lambda_0}( x \mid X_n )}{\sup_{x \in [-1, 1]} \mathbb{V}_{\lambda}( x \mid X_n )} \lesssim_n n^{-1/4} e^{2\sqrt{n} / \lambda_0} \bigg( \frac{\lambda^2}{\lambda_0^2} \bigg)^n
\end{equation}
tends to zero if $\lambda_0 > \lambda$ and explodes if $\lambda_0 < \lambda$.
However, Equation~\eqref{eq:gaussian-kernel-sup-ratio} is clearly not sufficient to establish either~\eqref{eq:variance-decay-interval-ml} or~\eqref{eq:variance-decay-interval-cv}.

\section*{Acknowledgements}

This work was supported by the Academy of Finland postdoctoral researcher grant \#338567 ``Scalable, adaptive and reliable probabilistic integration''.
I am grateful to the two reviewers for comments that resulted in substantial improvements, including results for self-similar functios.

\end{document}